\documentclass[reqno,11pt, draft]{amsart}

\usepackage[english]{babel}

\usepackage[all]{xy}
\usepackage[bookmarksnumbered,colorlinks]{hyperref}
\usepackage{dsfont}
\usepackage[normalem]{ulem}
\usepackage{amsmath}
\usepackage{graphicx}

\usepackage[euler-digits]{eulervm}
\usepackage{enumerate}
\def\br#1\er{\textcolor{red}{#1}} %
 \def\br#1\er{\textcolor{red}
 {#1}} %

  \def\bb#1\eb{\textcolor{blue}
 {#1}} %

\newcommand{\R}{\mathds R}

\newcommand{\germ}{\mathrm{germ}}

\usepackage{hyperref}

\hypersetup{
    pdftoolbar=true,
    pdfmenubar=true,
    pdffitwindow=false,
    pdfstartview={FitH},
    pdftitle={On a monodromy theorem for sheaves of local fields and applications},
    pdfauthor={Jonatan Herrera, Miguel Angel Javaloyes, Paolo Piccione},
    pdfsubject={},
    pdfkeywords={}
    pdfnewwindow=true,
    colorlinks=true, 
    linkcolor=blue,
    citecolor=blue,
    urlcolor=blue,
}
\numberwithin{equation}{section}

\title[Global extension property for sheaves of local fields]{On a monodromy theorem for sheaves of local fields and applications}
\thanks{This activity is supported by the programme ``Young leaders in research'' 18942/JLI/13  by Fundaci\'on S\'eneca, Regional Agency for Science and Technology from the Region of Murcia.   This work has been finsihed during a visit of the second author to the University of S\~ao Paulo supported by the Fapesp grant for visiting researchers with process 2013/1070-7. The first author is partially supported by Fapesp (Funda\c{c}\~ao de Amparo \'a Pesquisa do Estado de S\~ao Paulo, Brazil), Process 2012/11950-7; MICINN-FEDER project MTM2010--18099 and grants FQM-324, P09-FQM-4496, (J. Andaluc\'ia, the latter also with FEDER funds), Spain. The third author is partially sponsored by CNPq and Fapesp, Brazil.}
\date{July 12th, 2015}

\author[J. Herrera]{Jonatan Herrera}
\address{Departamento de Matem\'atica,\hfill\break\indent
Universidade de S\~ao Paulo, \hfill\break\indent Rua do Mat\~ao
1010,\hfill\break\indent CEP 05508-900, S\~ao Paulo, SP, Brazil}
\email{jonatanhf@gmail.com}

\author[M. A. Javaloyes]{Miguel Angel Javaloyes}
\address{Departamento de Matem\'aticas, \hfill\break\indent
Universidad de Murcia, \hfill\break\indent
Campus de Espinardo,\hfill\break\indent
30100 Espinardo, Murcia, Spain}
\email{majava@um.es}

\author[P. Piccione]{Paolo Piccione}
\address{Departamento de Matem\'atica,\hfill\break\indent
Universidade de S\~ao Paulo, \hfill\break\indent Rua do Mat\~ao
1010,\hfill\break\indent CEP 05508-900, S\~ao Paulo, SP, Brazil}
\email{piccione.p@gmail.com}

\begin{document}
\newtheorem{thm}{Theorem}[section]
\newtheorem{prop}[thm]{Proposition}
\newtheorem{lemma}[thm]{Lemma}
\newtheorem{cor}[thm]{Corollary}
\newtheorem{exe}[thm]{Example}
\theoremstyle{definition}
\newtheorem{defi}[thm]{Definition}
\newtheorem{notation}[thm]{Notation}
\newtheorem{conj}[thm]{Conjecture}
\newtheorem{prob}[thm]{Problem}
\newtheorem{rem}[thm]{Remark}
\maketitle

\begin{abstract}
We prove a monodromy theorem for local vector fields belonging to a sheaf satisfying the unique continuation property.
In particular, in the case of \emph{admissible regular} sheaves of local fields defined on a simply connected manifold, we obtain a global extension result for \emph{every} local field of the sheaf. This generalizes previous works of Nomizu \cite{nomizu60}
for semi-Riemannian Killing fields, of Ledger--Obata \cite{LedOba70} for conformal fields, and of Amores \cite{amores79} for fields preserving a $G$-structure of finite type. The result applies to Finsler or pseudo-Finsler Killing fields and, more generally, to affine fields of a spray. Some applications are discussed.
\end{abstract}

\renewcommand{\contentsline}[4]{\csname nuova#1\endcsname{#2}{#3}{#4}}
\newcommand{\nuovasection}[3]{\hbox to \hsize{\vbox{\advance\hsize by -1cm\baselineskip=12pt\parfillskip=0pt\leftskip=3.0cm\noindent\hskip -2.5cm \begin{small}#1\end{small}\leaders\hbox{.}\hfil\hfil\par}$\,$#2\hfil}}
\newcommand{\nuovasubsection}[3]{\hbox to \hsize{\vbox{\advance\hsize by -1cm\baselineskip=12pt\parfillskip=0pt\leftskip=3.3cm\noindent\hskip -2.5cm \begin{Small}#1\end{Small}\leaders\hbox{.}\hfil\hfil\par}$\,$#2\hfil}}
\newcommand{\nuovasubsubsection}[3]{\hbox to \hsize{\vbox{\advance\hsize by -1cm\baselineskip=12pt\parfillskip=0pt\leftskip=3.5cm\noindent\hskip -2.5cm \begin{Small}#1\end{Small}\leaders\hbox{.}\hfil\hfil\par}$\,$#2\hfil}}

\tableofcontents

\section{Introduction}

The question of global extension of locally defined Killing, or conformal fields of a semi-Riemannian manifold
appears naturally in several contexts. For instance, the property of global extendability of local Killing fields in simply connected Lorentzian manifolds plays a crucial role in the celebrated result that the isometry group of a compact simply connected analytic Lorentz manifold is compact, see \cite{Dambra}.
Nomizu's result in \cite{nomizu60} is often used in connection with de Rham splitting theorem to obtain that a real-analytic complete Riemannian manifold $M$ admitting a local parallel field has universal covering $\widetilde M$ which splits metrically as $N\times\mathds R$, see \cite[Theorem~C]{BetSch14} for an example.
Recently, the question of (unique) extensibility of local Killing vector fields has been studied in the context of Ricci-flat Lorentz manifolds, in the setting of the black hole rigidity problem, see \cite{IonKlai13}. The Killing extension property is also one of the crucial steps in the proof of the extensibility of locally homogeneous Lorentz metrics in dimension $3$, see \cite{DumMel14}.
It is to be expected that an analogous extension property should play an important role also for more general sheaves of local vector fields defined on a differentiable manifold.
The purpose of the present paper is to explore this question in geometric problems that have not been studied in the literature as of yet, like for instance in Finsler or pseudo-Finsler manifolds.

For a brief history of the problem, let us recall that the first result in the literature, concerning the existence of a global extension of local Killing fields defined on a simply-connected real-analytic Riemannian manifold, was proved by Nomizu in \cite{nomizu60}. The same proof applies to the semi-Riemannian case. The question of extension of local conformal fields has been studied in Ledger--Obata \cite{LedOba70}. A more conceptual proof of the extendability property was proved by Amores in \cite{amores79} for local vector fields whose flow preserves a $G$-structure \emph{of finite type} of an $n$-dimensional manifold $M$; here $G$ is a Lie subgroup of $\mathrm{GL}(n,\mathds R)$.
When $G$ is the orthogonal or the conformal group of some nondegenerate metric on $\mathds R^n$, then Amores' result reproduces the classic semi-Riemannian or conformal case of Nomizu and Ledger--Obata.

Although the class of $G$-structure automorphisms is quite general, restriction to those of finite type leaves several important examples of geometric structures out of the theory developed in \cite{amores79}. For instance, a Finsler structure on a manifold cannot be described in terms of a $G$-structure of finite type.
Moreover, it is interesting to observe that the sheaf of local vector fields whose flow preserves any $G$-structure has always a Lie algebra structure (see Proposition~\ref{thm:Liealgebrastr}), which seems to be unnecessary for the extendability problem we aim at. Thus, it is natural to extend the results of \cite{amores79, LedOba70, nomizu60}
to arbitrary vector space valued sheaves of local vector fields.
\smallskip

The classical standard reference for global extensions of local objects is found in elementary Complex Analysis, and it is given by the monodromy theorem for holomorphic functions, see for instance \cite[\S 10.3]{Krantz99}. The two essential ingredients for this theorem are:
\begin{itemize}
\item the unique continuation property of holomorphic functions on the complex plane;
\item a suitable notion of \emph{transport} along curves of (germs) of holomorphic functions, which is fixed-endpoint homotopy invariant.
\end{itemize}
The starting point of the theory developed here is the observation that these two ingredients can be reproduced in the more general context of sheaves of local vector fields on an arbitrary differentiable manifold (Definition~\ref{thm:deftransport}). This leads readily to the formulation of an abstract monodromy theorem for a sheaf $\mathcal F$ of vector fields satisfying the unique continuation property, which is discussed in Section~\ref{sec:abstractheory}. The result gives a necessary and sufficient condition for the existence of an extension to a prescribed open set for a given locally defined vector field  (or, equivalently, for a given germ of vector fields) in $\mathcal F$. The condition is given in terms of existence of $\mathcal F$-transport of the germ along a family of curves (Proposition~\ref{thm:monodromyuniquecont}).

A different type of question deals with the existence of global extensions of \emph{all} $\mathcal F$-germs, which is a property of the underlying geometric structure associated with the sheaf $\mathcal F$. Since transport is only a homotopical invariant, it is clear that one has topological obstructions arising from the topology of the manifold, which in our theory has to be assumed simply connected. Having settled this, the next crucial observation is that a necessary condition for the extension property in some sheaf $\mathcal F$ of local vector fields on a manifold $M$, is that the space of $\mathcal F$-germs must have the same dimension at each point of $M$. Namely, if the extension property holds, then the space of germs at each point has the same dimension as the space of globally defined vector fields of the given class. This condition is called \emph{regularity} in the paper (Definition~\ref{thm:defregular}); the reader should be warned that this notion of regularity for sheaves of vector fields is different from the one given in reference\footnote{%
See Section~\ref{sub:Fspecial} for the definition of regularity in reference \cite{SinSte65}.}
\cite[Definition~1.3]{SinSte65}. This leads naturally to a restriction to those sheaves of local fields whose germs at each point span a finite dimensional vector space whose dimension is uniformly bounded: these will be called \emph{bounded rank sheaves}.
Our main abstract result (Theorem~\ref{thm:globalregular}) is that, for a certain class of sheaves of local vector fields, called \emph{admissible}, defined on a simply connected manifold, then the regularity condition is also sufficient to guarantee that local fields extend (uniquely) to globally defined vector fields. By admissible, we mean sheaves of vector fields that have two basic properties: bounded rank and unique continuation, see Subsection~\ref{sub:admissible}.

A proof of Theorem~\ref{thm:globalregular} is obtained by showing that, for an admissible sheaf $\mathcal F$ of local fields, the regularity assumption guarantees the existence of transport of any
$\mathcal F$-germ along any continuous path (Proposition~\ref{thm:existencetransport}), and then applying the monodromy theorem (Proposition~\ref{thm:monodromyuniquecont}).
It is interesting to observe that a somewhat similar proof, using a notion of transport of $1$-jets, had been employed in the original article of Nomizu \cite{nomizu60}, using a differential equation along smooth curves satisfied by Killing vector fields. In the abstract framework considered in the present paper, no such differential equation is available, and one has to resort to purely topological arguments.

As to the regularity assumption, which plays a fundamental role in the statement of Theorem~\ref{thm:globalregular}, in Section~\ref{sec:regularity} we describe two situations in which it is satisfied. A first important case where regularity holds is when the sheaf is real-analytic, and defined by some differential equation with real-analytic coefficients, see Definition~\ref{def:realanaldet} and Proposition~\ref{thm:ADimpliesreg}.
This applies in particular to the sheaf of vector fields whose flow preserves a real-analytic $G$-structure of finite type (Theorem~\ref{thm:FPanalyticregular}).
A second class of examples of regular sheaves is described in terms of the transitivity of the action of the associated pseudo-group of local diffeomorphisms, see Proposition~\ref{thm:transitiveimpliesregular}.
For sheaves of Lie algebras, transitivity is equivalent to the weaker property, often easier to check, that the $0$-jets
of the sheaf span the whole tangent bundle, see Proposition~\ref{thm:transitivities}.

The statement of Theorem~\ref{thm:globalregular} generalizes Nomizu's result \cite{nomizu60}
for semi-Riemannian Killing fields, of Ledger--Obata \cite{LedOba70} for conformal fields, and of Amores \cite{amores79} for fields preserving a $G$-structure of finite type
($G$-fields). Namely, the bounded rank and unique continuation property are easily proven in the case of semi-Riemannian Killing fields and in the case of conformal fields; we show here that they hold in the general case of vector fields whose flow consists of local automorphisms of a $G$-structure of finite type (Section~\ref{sec:fieldsGstructureFinType}). It is worth mentioning that the case of $G$-structures of finite type is studied here with methods different than the ones employed in \cite{amores79}. More precisely, in the present paper we determine an explicit characterization of $G$-fields in terms of compatible connections (see Proposition~\ref{thm:charGfields}), which seems to have an independent interest. This characterization is given by a first order equation involving a compatible connection and its torsion, and it admits a higher order generalization to characterize the jets of any order of a $G$-field. A very interesting observation that can be drawn from this approach is that the unique continuation property for local $G$-fields has nothing to do with the analyticity. On the one hand, it holds for any $G$-structure of finite order (see Proposition~\ref{thm:Nthjetvanishing} and Corollary~\ref{thm:fincontpropfinorderGstruct}); on the other hand, it may fail to hold for analytic $G$-structures that have infinite order (see Remark~\ref{rem:locuniqvsanalyt}).

The result of Theorem~\ref{thm:globalregular}, however, goes well beyond the case of $G$-structures of finite type, considered in \cite{amores79}.
The main example discussed in this paper is the sheaf of local affine fields of a spray, see Section~\ref{sec:sprays}.
Recall that a local field on a manifold $M$ endowed with a spray $S$ is an affine field of $(M,S)$ if its flow preserves the geodesics of $S$.
As a preparatory case, we discuss two important special cases of sprays: Finsler and (conic) pseudo-Finsler structures, see Section~\ref{sec:KillingFinsler} and Section~\ref{sec:pseudoFinsler}. In the case of
Finsler manifolds, we study also local conformal and almost Killing fields. The special features of these two examples are investigated employing auxiliary semi-Riemannian structures: the average Riemannian metric associated with Finsler metrics, and the Sasaki metric associated with (conic) pseudo-Finsler manifolds.
In Section~\ref{sec:sprays}, we will give a differential characterization of affine fields of a spray in terms of the Berwald connection (Proposition~\ref{thm:charaffinefields}), and this will be used to show the admissibility of the sheaf (Proposition~\ref{thm:affineadmissible}). Regularity is also obtained in the analytic case using the differential equation satisfied by the affine fields (Proposition~\ref{regularaffine}).
\section{An abstract extension problem for sheaves of local vector fields}
\label{sec:abstractheory}
We will discuss here an abstract theory concerning the global extension property for sheaves of vector fields
on differentiable manifolds. Some of the ideas employed here have their origin in reference \cite{nomizu60}, where similar results were obtained in the case of Killing vector fields in Riemannian manifolds, although our approach does not use differential equations. A central notion for our theory is a property called unique continuation for sheaves of local vector fields on a manifold.
\subsection{Sheaves with the unique continuation property}
Let $M$ be a differentiable manifold with $\mathrm{dim}(M)=n$, and for all open subset $U\subset M$, let
$\mathfrak X(U)$ denote the sheaf of smooth vector fields on $U$. By smooth we mean of class $C^\infty$, although many of the results of the present paper hold under less regularity assumptions.
Given $V\subset U$ and $K\in\mathfrak X(U)$, we will denote by $K\mapsto K\big\vert_V\in\mathfrak X(V)$ the
restriction map.

For every connected open subset $U$, let $\mathcal F(U)\subset\mathfrak X(U)$ be a subspace, and assume that this family is a vector sub-sheaf
of $\mathfrak X$, namely, if $K\in \mathcal F(U)$ and $V\subset U$, then $K|_V\in \mathcal F(V)$ and if a vector field $K\in \mathfrak X(U)$, $U=\bigcup_{i\in I} U_i$ with $K|_{U_i}\in \mathcal F(U_i)$, then $K\in \mathcal F(U)$.
Elements of $\mathcal F$ will be called \emph{local $\mathcal F$-fields}.
\smallskip

We say that $\mathcal F$ has the \emph{unique continuation property} if,
given an open connected set $U\subset M$ and given $K\in\mathcal F(U)$, there exists a non-empty open subset $V\subseteq U$ such that $K|_V=0$, then $K=0$.

The unique continuation property implies that, given connected open sets $U_1\subset U_2\subset M$, then the map
$\mathcal F(U_2)\ni K\mapsto K\big\vert_{U_1}\in\mathcal F(U_1)$ is injective, and therefore one has a monotonicity property:
\begin{equation}\label{eq:monotonicity}
\mathrm{dim}\big(\mathcal F(U_1)\big)\ge \mathrm{dim}\big(\mathcal F(U_2)\big),\qquad\text{for $U_1\subset U_2$ connected.}
\end{equation}
Given $p\in M$, we define $\mathfrak G^\mathcal F_p$ as the space of \emph{germs at $p$} of local $\mathcal F$-vector fields\footnote{%
The space $\mathfrak G^\mathcal F_p$ is also called the \emph{stalk} of the sheaf $\mathcal F$ at $p$.},
also called \emph{$\mathcal F$-germs at $p$}.
More precisely, $\mathfrak G^\mathcal F_p$ is defined as the quotient of the set:
\[\bigcup_{U}\mathcal F(U),\]
where $U$ varies in the family of connected open subsets of $M$ containing $p$, by the following equivalence relation $\cong_p$:
for $K_1\in\mathcal F(U_1)$ and $K_2\in\mathcal F(U_2)$, $K_1\cong_p K_2$ if $K_1=K_2$ on a non-empty connected open subset  $V\subseteq U_1\cap U_2$.
For $K\in\mathcal F(U)$, with $p\in U$, we will denote by $\germ_p(K)$ the germ of $K$ at $p$, which is
defined as the $\cong_p$-equivalence class of $K$.
It is easy to see that $\mathfrak G_p^\mathcal F$ is a vector space for all $p$; moreover, for all $p\in M$ and all connected open neighborhood $U$ of $p$, the map:
\begin{equation}\label{eq:germmap}
\mathcal F(U)\ni K\longmapsto\germ_p(K)\in\mathfrak G^\mathcal F_p
\end{equation}
is linear  and injective, since if $\germ_p(K)=0$, then by definition there exists $V\subseteq U$ such that $K|_V=0$ and by the unique continuation property, we have that $K=0$.  In particular, if two local $\mathcal F$-fields have the same germ
at some $p\in M$, they coincide in any connected open neighborhood of $p$ contained in their common domain.

We will also use the evaluation map:
\begin{equation}\label{eq:evaluationmap}
\mathrm{ev}_p:\mathfrak G^\mathcal F_p\longrightarrow T_pM,
\end{equation}
defined by $\mathrm{ev}_p(\mathfrak g)=K_p$, where $K$ is any local $\mathcal F$-field around $p$ such that
$\germ_p(K)=\mathfrak g$.
\subsection{Transport of germs of $\mathcal F$-fields}
We will henceforth assume that $\mathcal F$ is a sheaf of local vector fields that satisfies the unique continuation property.
We will now  define the notion of \emph{transport} of
germs of $\mathcal F$-vector fields along curves. The reader should note the analogy between this definition
and the classical notion of analytic continuation along a curve in elementary complex analysis.
\begin{defi}\label{thm:deftransport}
Let $\gamma:[a,b]\to M$ be a curve, and let $\mathfrak g$ be a distribution of $\mathcal F$-germs along $\gamma$, i.e., a map $[a,b]\ni t\mapsto \mathfrak g_t\in\mathfrak G^\mathcal F_{\gamma(t)}$.
We say that $\mathfrak g$ is a \emph{transport of $\mathcal F$-germs along $\gamma$} if the following holds: for all $t_*\in[a,b]$, there exists $\varepsilon>0$, an open neighborhood $U$ of $\gamma(t_*)$ and $K\in\mathcal F(U)$ such that $\gamma(t)\in U$ and $\mathfrak g_t=\germ_{\gamma(t)}(K)$ for all $t\in\left(t_*-\varepsilon,t_*+\varepsilon\right)\cap[a,b]$.
\end{defi}
It is easy to see that initial conditions identify uniquely transports of $\mathcal F$-germs along continuous
curves.
\begin{lemma}\label{thm:uniquetransport}
If $\mathfrak g^{(1)}$ and $\mathfrak g^{(2)}$ are transports of $\mathcal F$-germs along
a continuous path $\gamma:[a,b]\to M$ such that $\mathfrak g^{(1)}_{t_*}=\mathfrak g^{(2)}_{t_*}$
for some $t_*\in[a,b]$, then $\mathfrak g^{(1)}=\mathfrak g^{(2)}$.
\end{lemma}
\begin{proof}
The set $\mathcal A=\big\{t\in[a,b]:\mathfrak g^{(1)}_t=\mathfrak g^{(2)}_t\big\}$ is open, by the very definition of transport, and non-empty, because $t_*\in\mathcal A$. Let us show that it is closed.
Denote by $\overline{\mathcal A}$ the closure of $\mathcal A$ in $[a,b]$, and let $t_0\in\overline{\mathcal A}$;
let $U$ be a sufficiently small connected open neighborhood of $\gamma(t_0)$.
There exist $K_1,K_2\in\mathcal F(U)$ such that, for $t$ sufficiently close to $t_0$:
\begin{equation}\label{eq:equalitygerms}
\germ_{\gamma(t)}K_1=\mathfrak g_t^{(1)}\quad\text{and}\quad\germ_{\gamma(t)}K_2=\mathfrak g_t^{(2)}.
\end{equation}
Since $t_0$ is in the closure of $\mathcal A$, we can find $t\in\mathcal A$ such that $\gamma(t)\in U$ and
for which \eqref{eq:equalitygerms} holds; in particular, the germs of $K_1$ and $K_2$ at such $t$ coincide.
It follows by the unique continuation property that $K_1=K_2$ in $U$, and therefore $\mathfrak g_{t_0}^{(1)}=\germ_{\gamma(t_0)}K_1=\germ_{\gamma(t_0)}K_2=\mathfrak g_{t_0}^{(2)}$, i.e.,
$t_0\in\mathcal A$ and $\mathcal A$ is closed.
\end{proof}
\subsection{Homotopy invariance}\label{sub:homotopyinv}
We will now show that the notion of transport of germs is invariant by fixed endpoints homotopies.
\begin{prop}\label{thm:homotopyinv}
Let $\mathcal F$ be a sheaf of local fields on a manifold $M$; assume that $\mathcal F$ satisfy the unique continuation property. Let $\gamma:[0,1]\times[a,b]\to M$ be a continuous map such that $\gamma(s,a)=p$ and $\gamma(s,b)=q$ for all $s\in[0,1]$, with $p,q\in M$ two fixed points. For all $s\in[0,1]$ set $\gamma_s:=\gamma(s,\cdot):[a,b]\to M$.
Fix $\mathfrak g_a\in\mathfrak G^\mathcal F_p$, and assume that for all $s\in[0,1]$ there exists a transport $\mathfrak g^{(s)}$ of $\mathcal F$-germs along $\gamma_s$, with $\mathfrak g^{(s)}_a=\mathfrak g_a$ for all $s$.
Then, $\mathfrak g^{(s)}_b\in\mathfrak G^\mathcal F_q$ does not depend on $s$.
\end{prop}
\begin{proof}
It suffices to show that $s\mapsto\mathfrak g^{(s)}_b$ is locally constant on $[0,1]$.
Fix $s\in[0,1]$ and
choose $N\in\mathds N$, $U_1,\ldots U_N$ connected open subsets of $M$, and $K_i\in\mathcal F(U_i)$,
 for $i=1,\ldots,N$, such that, setting
$t_i=a+i\frac{b-a}N$, $i=0,\ldots,N$:
\begin{itemize}
\item $\gamma_s\big([t_{i-1},t_{i}]\big)\subset U_i$;
\item $\germ_{\gamma(t)}(K_i)=\mathfrak g_t$ for all $t\in[t_{i-1},t_i]$,
\end{itemize}
for all $i=1,\ldots,N$. Then, given $s'$ sufficiently close to $s$,  $\gamma_{s'}\big([t_{i-1},t_{i}]\big)\subset U_i$ for all $i=1,\ldots,N$. This implies that the
local fields $K_i$ can be used to define a transport $\mathfrak g^{(s')}$ of $\mathfrak g_a$
along $\gamma_{s'}$. In particular, $\mathfrak g^{(s')}_b=K_N\big(\widetilde\gamma(b)\big)=
K_N\big(\gamma(b)\big)=\mathfrak g^{(s)}_b$. This concludes the proof.
\end{proof}
\subsection{A monodromy theorem}\label{sub:monodromy}
We are now ready to prove the aimed extension of the classical monodromy theorem in Complex Analysis for sheaves of local vector fields satisfying the unique continuation property.
\begin{prop}\label{thm:monodromyuniquecont}
Let $\mathcal F$ be  a sheaf of local fields on a manifold $M$; assume that $\mathcal F$ satisfies the unique continuation property. Let $p\in M$ be fixed, let $U$ be a simply connected open neighborhood of $p$, and let $\mathfrak g_p\in\mathfrak G^\mathcal F_p$ be an $\mathcal F$-germ at $p$ such that for all continuous curve $\gamma:[a,b]\to U$ with $\gamma(a)=p$, there exists a transport of $\mathfrak g_p$ along $\gamma$. Then, there exists (a necessarily unique) $K\in\mathcal F(U)$ such that $\germ_p(K)=\mathfrak g_p$.
\end{prop}
\begin{proof}
Since $U$ is simply connected, any two continuous curves in $U$ between two given points are fixed endpoints homotopic. This allows one to define a vector field $K$ on $U$ as follows. Given $q\in U$, choose any continuous path $\gamma:[a,b]\to U$ with $\gamma(a)=p$ and $\gamma(b)=q$; let $\mathfrak g^\gamma$ be the unique transport of $\mathcal F$-germs along $\gamma$ such that $\mathfrak g^\gamma_a=\mathfrak g_p$, which exists by assumption. By Proposition~\ref{thm:homotopyinv}, $\mathfrak g_b^\gamma$ does not depend on $\gamma$. Then, we set $K_q:=\mathrm{ev}_q(\mathfrak g^\gamma_b)$, where $\mathrm{ev}$ is the evaluation map defined in \eqref{eq:evaluationmap}. It is easy to see that $K\in\mathcal F(U)$, using the definition of transport.
\end{proof}
\section{Admissible and regular sheaves}
We will now study a class of sheaves for which every germ can be transported along any curve. This class is defined in terms of two properties: rank boundedness and regularity.
\subsection{Sheaves with bounded rank}
\label{sub:admissible}
We say that $\mathcal F$ has \emph{bounded rank} if there exists
a positive integer $N^\mathcal F$ such that:
\[\mathrm{dim}\big(\mathcal F(U)\big)\le N^\mathcal F,\]
for all connected open subset $U\subset M$.
Let us call \emph{admissible} a sheaf which has both unique continuation and rank boundedness properties.
Examples of admissible sheaves will be discussed in Sections~\ref{sec:fieldsGstructureFinType}, \ref{sec:KillingFinsler}, \ref{sec:pseudoFinsler} and \ref{sec:sprays}.
\begin{rem}\label{rem:admissiblesubsheaf}
Observe that admissibility is inherited by sub-sheaves.
\end{rem}

\begin{lemma}\label{thm:exispecialneighborhoods}
Let $\mathcal F$ be an admissible sheaf of local fields on $M$. Then, given any $p\in M$, there exists
a connected open neighborhood $U$ of $p$ such that the map \eqref{eq:germmap}
is an isomorphism. In particular:
\[\mathrm{dim}(\mathfrak G^\mathcal F_p)=\mathrm{dim}\big(\mathcal F(U))\le N^\mathcal F.\]
\end{lemma}
\begin{proof}
Let us set:
\begin{equation}\label{eq:defkappaFp}
\kappa_p^\mathcal F:=\max\left\{\mathrm{dim}\big(\mathcal F(U)\big):U\ \text{connected open subset of $M$ containing $p$}\right\}.
\end{equation}
If $U$ is any connected neighborhood for which $\mathrm{dim}\big(\mathcal F(U)\big)=\kappa_p^\mathcal F$,
then the map \eqref{eq:germmap} is an isomorphism.
This follows easily from the  monotonicity property \eqref{eq:monotonicity}.
\end{proof}

From Lemma~\ref{thm:exispecialneighborhoods}, we have the following immediate semi-continuity result for the map $\kappa^\mathcal F$:
\begin{prop}\label{thm:lowerSCkappa}
Let $\mathcal F$ be an admissible sheaf of local fields on $M$. The map $M\ni p\mapsto\kappa_p^\mathcal F\in\mathds N$ is lower semi-continuous.
\end{prop}
\begin{proof}
Given any $p\in M$, let $\mathcal U$ be a connected open neighborhood of $p$ as in Lemma~\ref{thm:exispecialneighborhoods}. Then,
$\kappa_p^\mathcal F=\mathrm{dim}\big(\mathcal F(U)\big)\le\kappa_q^\mathcal F$ for all $q\in\mathcal U$.
\end{proof}

We will assume henceforth that $\mathcal F$ is an admissible sheaf of local fields on the manifold $M$.

\subsection{$\mathcal F$-special neighborhoods}
\label{sub:Fspecial}
A connected open neighborhood $U$ of $p$ as in Lemma~\ref{thm:exispecialneighborhoods} will be
said to be \emph{$\mathcal F$-special} for $p$. If $U$ is $\mathcal F$-special for $p$, then given
any $\mathfrak g_p\in\mathfrak G_p^\mathcal F$ there exists a unique $K\in\mathcal F(U)$ such that
$\germ_p(K)=\mathfrak g_p$. Observe that the number $\kappa^\mathcal F_p$ defined in \eqref{eq:defkappaFp}
is also given by:
\[\kappa_p^\mathcal F=\mathrm{dim}(\mathfrak G^\mathcal F_p).\]

\begin{lemma}\label{thm:special}
Let $\mathcal F$ be an admissible sheaf of local fields on $M$.
The following statements hold:
\begin{enumerate}[(i)]
\item\label{itm:1} If $U$ is $\mathcal F$-special for $p$, and $U'\subset U$ is a connected open subset containing $p$, then also $U'$ is
$\mathcal F$-special for $p$. In particular, every point admits \emph{arbitrarily small} $\mathcal F$-special connected open neighborhoods.
\smallskip

\item\label{itm:2} Assume that $p\in U$ and that $\kappa^\mathcal F_q$ is constant for $q\in U$. Then, there exists $U'\subset U$
containing $p$ which
is $\mathcal F$-special for all its points.
\end{enumerate}
\end{lemma}
\begin{proof}
For the proof of \eqref{itm:1}, note that if $p\in U'\subset U$, with $U'$ connected, then:
\[\kappa_p^\mathcal F\ge\dim\big(\mathcal F(U')\big)\stackrel{\text{by \eqref{eq:monotonicity}}}\ge\dim\big(\mathcal F(U)\big)=\kappa_p^\mathcal F,\]
i.e., $\kappa_p^\mathcal F=\dim\big(\mathcal F(U')\big)$, and $U'$ is $\mathcal F$-special for $p$.

As to the proof of \eqref{itm:2}, let $U'\subset U$ be any $\mathcal F$-special open connected neighborhood of $p$. Then, for all $q\in U'$:
\[\kappa_q^\mathcal F=\kappa_p^\mathcal F=\mathrm{dim}\big(\mathcal F(U')\big),\]
i.e., $U'$ is $\mathcal F$-special for $q$.
\end{proof}

The interest in open sets which are $\mathcal F$-special for all their points is the following immediate
consequence of the definition:
\begin{prop}\label{thm:immediate}
If $U$ is $\mathcal F$-special for all its points, then given any $p\in U$ and any $\mathfrak g_*\in \mathfrak G^\mathcal F_p$, there exists a unique
$K\in\mathcal F(U)$ such that $\germ_p(K)=\mathfrak g_*$.  Moreover, if $V\subset U$ is any connected open subset,
then for all $K\in\mathcal F(V)$ there exists a unique $\widetilde K\in\mathcal F(U)$ such that $\widetilde K\vert_V=K$.\qed
\end{prop}
\begin{defi}\label{thm:defregular}
A connected open subset $U$ of $M$ on which $\kappa^\mathcal F_p$ is constant
will be called \emph{$\mathcal F$-regular}. An admissible sheaf $\mathcal F$ will be called \emph{regular}\footnote{%
In \cite[Definition~1.3]{SinSte65}, a sheaf $\mathcal F$ is defined to be \emph{regular} when the map $\mathrm{ev}_p$ defined in \eqref{eq:evaluationmap} has constant rank for $p\in M$.
} if $\kappa^\mathcal F_p$ is constant for all $p\in M$.
\end{defi}
If $U$ is a connected open set which is $\mathcal F$-special for all its points, then $U$ is $\mathcal F$-regular. We will now determine under which circumstances the
converse of this statement holds.
\subsection{Existence of transport}
As to the existence of a transport of $\mathcal F$-germs with given initial conditions, one has
to assume regularity for the sheaf $\mathcal F$.

\begin{prop}
\label{thm:existencetransport}
Let $\gamma:[a,b]\to M$ be a continuous path, and assume that $\kappa^\mathcal F$ is constant along $\gamma$. Then, given any $\mathfrak g_*\in\mathfrak G^\mathcal F_{\gamma(a)}$
there exists a unique transport $[a,b]\ni t\mapsto \mathfrak g_t\in\mathfrak G^\mathcal F_{\gamma(t)}$
of $\mathcal F$-germs along $\gamma$ such that
\begin{equation}\label{eq:InCondTransport}
\mathfrak g_a=\mathfrak g_*.
\end{equation}
\end{prop}
\begin{proof}
Define:
\[\mathcal A=\Big\{t\in\left(a,b\right]:\exists\ \text{transport
of $\mathcal F$-germs  $\mathfrak g$ along $\gamma\vert_{[a,t]}$ satisfying \eqref{eq:InCondTransport}}\Big\}.\]
Let us show that $\mathcal A=\left(a,b\right]$. By Lemma~\ref{thm:exispecialneighborhoods}, there exists an open neighborhood $U_a$ of $\gamma(a)$ which is $\mathcal F$-special for $\gamma(a)$.
Then, there exists $K\in\mathcal F(U_a)$ such that $\germ_{\gamma(a)}K=\mathfrak g_*$.
This says that $\mathcal A$ contains an interval of the form $\left(a,\varepsilon\right]$, for some $\varepsilon>0$.

Arguing as in Lemma~\ref{thm:uniquetransport}, we can show that $\mathcal A$ is both open and closed in $\left(a,b\right]$. Again, openness follows readily from the very definition of transport.

In order to show that $\mathcal A$ is closed, assume that $t_0\in\overline{\mathcal A}$ (the closure of $\mathcal A$), and let
$U$ be an open neighborhood of $\gamma(t_0)$ which is $\mathcal F$-special for $\gamma(t_0)$.
Choose $t_1\in\mathcal A$ such that $\gamma(t_1)\in U$, then applying the uniqueness of Lemma \ref{thm:uniquetransport}, we get that there exists $K\in\mathcal F(U)$ such that
$\germ_{\gamma(t)}(K)=\mathfrak g_t$ for all $t$ sufficiently close to $t_1$.  This is because $U$ is also $\mathcal F$-special for all these points $\gamma(t)$ due to the assumption that $\kappa^\mathcal F$ is constant along $\gamma$.
Then, $\germ_{\gamma(t)}( K)=\mathfrak g_t$ for all $t$ sufficiently close to $t_0$  (more generally for all $t$ such that $\gamma(t)\in U$),  hence
$t_0\in\mathcal A$ and $\mathcal A$ is closed.
Thus, $\mathcal A=\left(a,b\right]$, and the existence of transport is proved.

Uniqueness follows directly from Lemma~\ref{thm:uniquetransport}.
\end{proof}

As a corollary, we can now state the main result of this section.

\begin{thm}\label{thm:globalregular}
Let $M$ be a (connected and) simply connected differentiable manifold, and let
$\mathcal F$ be an admissible and regular sheaf of local fields on $M$. Given any connected open subset $U\subset M$ and any $\widetilde K\in\mathcal F(U)$, there exists a unique $K\in\mathcal F(M)$ such that $K\big\vert_U=\widetilde K$. Similarly, given any $p\in M$ and any
$\mathfrak g_p\in\mathfrak G_p^\mathcal F$, then there exists a unique $K\in\mathcal F(M)$ such that
$\germ_p(K)=\mathfrak g_p$.
\end{thm}
\begin{proof}
The first statement follows easily from the second, which is proved as follows.
Given any $p\in M$, any germ $\mathfrak g_p\in\mathfrak G_p^\mathcal F$, and any continuous curve $\gamma$ in $M$ starting at $p$, by regularity (Proposition~\ref{thm:existencetransport}) there exists a transport of $\mathfrak g_p$ along $\gamma$. The existence of the desired field $K\in\mathcal F(M)$ follows from the monodromy theorem, Proposition~\ref{thm:monodromyuniquecont}.
\end{proof}
Theorem~\ref{thm:globalregular} applies in particular to the sheaf of fields whose flow preserves a real-analytic $G$-structure of finite type
(Proposition~\ref{thm:Gstructsheafboundedrank}, Corollary~\ref{thm:fincontpropfinorderGstruct}, Theorem~\ref{thm:FPanalyticregular}), giving a generalization of \cite[(0.2) Theorem]{amores79}.
\section{Sheaves of fields preserving a $G$-structure of finite type}
\label{sec:fieldsGstructureFinType}

\subsection{Infinitesimal symmetries of a $G$-structure}
\label{sub:Gstructuresinfsymmetries}
Let $G$ be a Lie subgroup of $\mathrm{GL}(n)$, denote by $\mathfrak g$ its Lie algebra, and let $M^n$ be a manifold with a $G$-structure $\mathcal P$. Recall that this is a $G$-principal sub-bundle of the frame bundle\footnote{By a \emph{frame} of an $N$-dimensional vector space $V$, we mean an isomorphism $q:\mathds R^N\to V$.}
$\mathrm{FR}(TM)$ of $TM$.
\subsubsection{$G$-fields}
It is well known that there exists a connection $\nabla$ compatible with the $G$-structure,\footnote{%
Such compatible connection is not symmetric in general. For instance, if $G$ is the complex general linear subgroup
$\mathrm{GL}(n,\mathds C)$, so that the corresponding $G$-structure is an almost complex structure, then there exists a symmetric compatible connection if and only if the structure is in fact complex.} i.e., such that the parallel transport of frames of $\mathcal P$ belong to $\mathcal P$.
A (local) vector field $K$ on $M$ is said to be an \emph{infinitesimal symmetry} of the $G$-structure $\mathcal P$, or a \emph{$G$-field},
if its flow preserves the $G$-structure. One can characterize $G$-fields using a compatible connection; to this aim, we need to recall a few facts about the lifting of vector fields to the frame bundle.
\subsubsection{Lifting $G$-fields to the frame bundle}\label{sub:lifting}
Let us recall a notion of lifting of (local) vector fields defined on a manifold $M$ to (local) vector fields
in the frame bundle $\mathrm{FR}(TM)$. First, diffeomorphisms of $M$ (or between open subsets of $M$) can be naturally  lifted to diffeomorphisms\footnote{Given a diffeomorphism $f$, the lifted map $\widetilde f$ is defined by $\widetilde f(p)=\mathrm df(x)\circ p$, where $p:\mathds R^n\to T_xM$ is a frame.} of $\mathrm{FR}(TM)$ (or between the corresponding open subsets of $\mathrm{FR}(TM)$); this notion of lifting preserves the composition of diffeomorphisms. Thus, given a vector field $X$ in $M$, the flow of $X$ can be lifted to a flow in $\mathrm{FR}(M)$; this is the flow of a vector field $\widetilde X$ in $\mathrm{FR}(TM)$.
An explicit formula for $\widetilde X$ can be written using a connection $\nabla$ on $M$; for this general computation,
we don't require that $\nabla$ is compatible with $\mathcal P$. Observe that $\nabla$ need not be symmetric either, and we will denote by $T$ its torsion.

If $F_t$ denotes the flow of $X$ and $\widetilde F_t$ its lifting, for fixed $x\in M$ and $p\in\mathrm{FR}(TM)$
 with $\pi(p)=x$, we have:
\[\widetilde X(p)=\tfrac{\mathrm d}{\mathrm dt}\big\vert_{t=0}\widetilde F_t(p)= \tfrac{\mathrm d}{\mathrm dt}\big\vert_{t=0}[\mathrm dF_t(x)\circ p]\]
for all $p\in\mathrm{FR}(TM)$ and all $x\in M$. Observe that $\widetilde F_t(p)=\mathrm dF_t(x)\circ p$ is a frame at the point $F_t(x)$.

The horizontal component of $\widetilde X(p)$ is:
\begin{multline}\label{eq:Xtilhor}
[\widetilde X(p)]_{\mathrm{hor}} =\mathrm d\pi_p\big(\tilde X(p)\big) = \tfrac{\mathrm d}{\mathrm dt}\big\vert_{t=0}\big[\pi(\mathrm dF_t(x)\circ p)\big]\\ = \tfrac{\mathrm d}{\mathrm dt}\big\vert_{t=0}[F_t(x)]|= X(x),
\end{multline}
i.e., $\widetilde X$ is indeed a lifting of $X$.
The vertical component of $\widetilde X(p)$ is:
\begin{equation}\label{eq:Xtilver}
[\widetilde X(p)]_{\mathrm{vert}} = \tfrac{\mathrm D}{\mathrm dt}\big\vert_{t=0}[\mathrm dF_t(x)\circ p]= \tfrac{\mathrm D}{\mathrm dt}\big\vert_{t=0}[\mathrm dF_t(x)] \circ p;
\end{equation}
for this equality, we have used the fact that  $\frac{\mathrm D}{\mathrm dt}$ commutes with right-com\-position with $p$.
We have
\[\tfrac{\mathrm D}{\mathrm dt}\big\vert_{t=0}[dF_t(x)v]= \tfrac{\mathrm D}{\mathrm dt}\big\vert_{t=0}\,\tfrac{\mathrm d}{\mathrm ds}\big\vert_{s=0}\big[F_t\big(x(s)\big)\big],\]
where $(-\varepsilon,\varepsilon)\ni s\mapsto x(s)\in M$ for $\varepsilon>0$, is a smooth curve such that $x(0)=x$ and $\dot x(0)=v$. Then:
\begin{multline*}
\tfrac{\mathrm D}{\mathrm dt}\big\vert_{t=0}\,\tfrac{\mathrm d}{\mathrm ds}\big\vert_{s=0}\big[F_t\big(x(s)\big)\big]\\ =\tfrac{\mathrm D}{\mathrm ds}\big\vert_{s=0}\,\tfrac{\mathrm d}{\mathrm dt}\big\vert_{t=0}\big[F_t\big(x(s)\big)\big] + T\left(\tfrac{\mathrm d}{\mathrm dt}\big[F_t\big(x(s)\big)\big],\tfrac{\mathrm d}{\mathrm ds}\big[F_t\big(x(s)\big)\big]\right)\big\vert_{s=t=0}
\\ = \nabla_v X(x) + T\big(X(x),v\big).
 \end{multline*}
Therefore, the vertical component of $\widetilde X(p)$ is:
\begin{equation}\label{eq:vercomptildeX}
[\widetilde{X}(p)]_{\mathrm{vert}} = \nabla X(x)\circ p + T\big(X(x),p\,\cdot\big).
\end{equation}
Using formulas \eqref{eq:Xtilhor} and \eqref{eq:Xtilver}, we can write:
\begin{equation}\label{eq:impobs}
\widetilde X=L(X\circ\pi),
\end{equation}
where $\pi:\mathrm{FR}(TM)\to M$ is the canonical projection and $L$ is a linear first order differential operator from sections of $\pi^*(TM)$ to sections of $T\big(\mathrm{FR}(TM)\big)$.
\subsubsection{Characterization of $G$-fields}\label{sub:charGfields}
For all $x\in M$, let $\mathfrak g_x$ be the Lie subalgebra of $\mathrm{gl}(T_xM)$ that corresponds to the Lie group
$G_x$ of automorphisms of the $G$-structure of the tangent space $T_xM$. Clearly, $\mathfrak g_x$ is isomorphic to $\mathfrak g$ for all $x$. Using the lifting of vector fields to the frame bundle, it is not hard to prove the following:
\begin{prop}\label{thm:charGfields}
Let $\nabla$ be a connection on $M$ compatible with the $G$-structure $\mathcal P$, and let $T$ denote its torsion.
A (local) vector field $K$ is a $G$-field
if and only if
\begin{equation}\label{eq:condGfield}
\nabla K(x)+T(K,\cdot)\in\mathfrak g_x
\end{equation}
for\footnote{i.e., the map $T_xM\ni v\mapsto\nabla_vK+T(K_x,v)\in T_xM$ belongs to $\mathfrak g_x$.} all $x$ in the domain of $K$.
\end{prop}
\begin{proof}
$K$ is a $G$-field if and only if the flow of its lifting $\widetilde K$ preserves the $G$-structure $\mathcal P$,
that is, if and only if $\widetilde K$ is everywhere tangent to $\mathcal P$.
Using the fact that $\nabla$ is compatible with $\mathcal P$, one sees that the tangent space to $\mathcal P$ at some point $p\in\mathcal P$ is given by the space of vectors whose vertical component has the form $h\circ p$, with $h\in\mathfrak g_x$.
Therefore, using formula \eqref{eq:vercomptildeX} for the vertical component of liftings, we have
that $\widetilde K$ is tangent to $\mathcal P$ if and only if \eqref{eq:condGfield} holds
for all $x$.
\end{proof}
\subsubsection{Admissibility of the sheaf of $G$-fields for $G$-structures of finite type}
\begin{defi}\label{thm:defsheafGfields}
Given an open subset $U\subset M$, let $\mathcal F^\mathcal P(U)$ be the subspace of $\mathfrak X(U)$ consisting of all vector fields
$K$ satisfying $\nabla K(x)+ T\big(K(x),.\big)\in\mathfrak g_x$ for all $x\in U$. This is a sheaf of vector fields in $M$ that will be called
the sheaf of \emph{$G$-vector fields} (shortly, \emph{$G$-fields}).
\end{defi}
It is not hard to prove that the sheaf of $G$-fields consists of Lie algebras. We will give an indirect proof of this fact in Proposition~\ref{thm:transitiveimpliesregular}.
Given $x\in M$, for all $j\ge1$, define the \emph{$j$-th prolongation $\mathfrak g_x^{(j)}$} of $\mathfrak g_x$ as the Lie algebra:
\begin{multline}\label{eq:defprolongation}
\mathfrak g_x^{(j)}=\big\{L:T_xM^{(j+1)}\longrightarrow T_xM\ \text{multilinear and symmetric}:\\ \forall\,(v_1,\ldots,v_j)\in T_xM^{(j)},\\
\text{the map}\ T_xM\ni v\mapsto L(v,v_1,\ldots,v_j)\in T_xM\ \text{belongs to}\ \mathfrak g_x \big\}.
\end{multline}
In particular, $\mathfrak g_x^{(0)}=\mathfrak g_x$. The $G$-structure is said to be of \emph{finite type} if $\mathfrak g_x^{(j)}=\{0\}$ for some
$j\ge1$; this condition is independent of $x$. The \emph{order} of a finite type $G$ structure is the minimum $j$ for which
$\mathfrak g_x^{(j)}=\{0\}$. For instance, when $G$ is some orthogonal group $\mathrm O(n,k)$ (in which case a $G$-structure is a semi-Riemannian
structure in $M$), then the $G$-structure is of finite type and it has order $1$. If the group $G$ is $\mathrm{CO}(n,k)$ (in which case a $G$-structure is a semi-Riemannian conformal structure in $M$), is of finite type and it has order $2$.
\begin{prop}\label{thm:Gstructsheafboundedrank}
Let $\mathcal P$ be a finite type $G$-structure on $M^n$, with order $N\ge1$. Then, the sheaf $\mathcal F^\mathcal P$ has rank bounded by $n+\sum_{j=0}^{N-1}\mathrm{dim}\big(\mathfrak g_x^{(j)}\big)$.
\end{prop}
\begin{proof}
It follows immediately from the corresponding statement on the dimension of the automorphism group of a $G$-structure of finite type. This is a very classical result, first stated in \cite{Ehr58}, see also
\cite[Corollary~4.2, p.\ 348]{Sternberg}, \cite{Ruh64}, or \cite[Theorem~1, p.\ 333]{KobaNomi2}.
\end{proof}
Proposition~\ref{thm:Gstructsheafboundedrank} holds also in the case of $G$-structures on orbifolds, see \cite{BagZhu03}.
\begin{prop}\label{thm:Nthjetvanishing}
Let $G\subset\mathrm{GL}(n)$, let $\mathcal P$ be a  $G$-structure of order $N$ on a differentiable manifold $M^n$, and let $X$ be a $G$-field.
If the $N$-th order jet of $X$ vanishes at some point $x\in M$, then $X$ vanishes in a neighborhood of $x$.
If $M$ is connected, then $X$ vanishes identically.
\end{prop}
\begin{proof}
The proof is by induction on $N$.

When $N=0$, then $G=\{1\}$ and $\mathcal P$ is a global frame $(X_1,\ldots,X_n)$ of $M$. In this case, the property of
being a $G$-field for $X$ means that $X$ commutes with each one of the $X_i$'s, i.e., $X$ is invariant by the flow of each $X_i$. In this case, the set of zeroes of $X$ is invariant by the flow of the $X_i$'s. Now, if $F^{X_i}_t$ denotes the flow of $X_i$ at time $t$, then, for fixed $x\in M$, the map:
\[
(t_1,...,t_n) \longmapsto F^{X_1}_{t_1}\circ\ldots\circ F^{X_n}_{t_n}(x)
\]
is a diffeomorphism from a neighborhood of $0$ in $\mathds R^n$ onto a neighborhood of $x$ in $M$.
In order to prove this, it suffices to observe that its derivative at $0$ sends the canonical basis of $\mathds R^n$ to $X_1(x),\ldots,X_n(x)$. Thus, if $X(x)=0$ and $X$ is invariant by the flows $F^{X_i}$, then $X$ vanishes in a neighborhood of $x$.

For the induction step, we need the notion of lifting of vector fields in $M$ to vector fields in the frame bundle $\mathrm{FR}(M)$ that was recalled in Subsection~\ref{sub:lifting}. Consider the lift $\widetilde X$ of $X$ to a vector field in $\mathrm{FR}(M)$.
From formula \eqref{eq:impobs} it follows readily that if the $N$-th jet of $X$ vanishes at $x$, then the
jet of order $(N-1)$ of $\widetilde X$ vanishes at every $p\in\mathrm{FR}(T_xM)$.

Observe that if $X$ is a $G$-field, then $\widetilde X$ is tangent to $\mathcal P$. Furthermore, if some diffeomorphism of $M$ preserves $\mathcal P$, then the lifting of such diffeomorphism to $\mathrm{FR}(TM)$ restricted to $\mathcal P$ preserves the $G^{(1)}$-structure $\mathcal P^{(1)}$ of $\mathcal P$, where $\mathcal P^{(1)}$ is the first prolongation of $\mathcal P$ (see \cite[Page 22]{koba72}). It follows that the restriction of $\widetilde X$ to $\mathcal P$ is a $G^{(1)}$-field for the $G^{(1)}$-structure $\mathcal P^{(1)}$.

Now, we can apply induction. Assume that $\mathcal P$ has order $k$, and that $X$ is a $G$-field for $\mathcal P$ with vanishing $N$-th jet at $x$.
The field $\widetilde X$ restricted to $\mathcal P$ is a $G^{(1)}$-field for $\mathcal P^{(1)}$ which has vanishing jet of order $(N-1)$ at $p$, for arbitrary $p\in\mathrm{FR}(T_xM)$. Since $\mathcal P^{(1)}$ has order $N-1$, the induction  hypothesis gives us that $\widetilde X$ vanishes identically in a neighborhood of $p$ in $\mathcal P$. It follows that $X$ vanishes in a neighborhood of $x$ in $M$ (the projection onto $M$ of the neighborhood of $p$). This concludes the proof of the first statement.

In particular, the result proves that the set of points at which the $N$-th order jet of $X$ vanishes is open.
Evidently, it is also closed, and therefore if $M$ is connected, then $X$ vanishes identically.
\end{proof}

\begin{cor}\label{thm:fincontpropfinorderGstruct}
The sheaf of local $G$-fields of a finite order $G$-structure has the unique continuation property.\qed
\end{cor}
\begin{rem}[Local uniqueness vs.\ analyticity]
\label{rem:locuniqvsanalyt}
It is an interesting question to establish for which $G$-structures of infinite type, the corresponding sheaf of local $G$-fields has the unique continuation property. A natural guess would be to consider real-analytic structures, this seems to have been claimed in \cite[Proposition~3.2, p.\ 5]{amores79}. However, there exist real-analytic $G$-structures of infinite type whose sheaf of local $G$-fields does not have the unique continuation property. Consider for instance a real-analytic symplectic manifold $(M^{2n},\omega)$. In this case, $G$ is the symplectic group $\mathrm{Sp}(2n)$.
Given any smooth local function $H:U\subset M\to\mathds R$, then its Hamiltonian field $\vec H$ has flow that preserves $\omega$. In particular, two distinct smooth functions that coincide on some non-empty open subset define distinct Hamiltonian fields, that coincide on that open set.
See also Remark~\ref{thm:notfintypenotal}.

Observe that, by Corollary~\ref{thm:fincontpropfinorderGstruct}, for $G$-structures of finite type, local uniqueness is independent of analyticity, thus, the two properties seem to be logically independent.
\end{rem}
\subsection{Affine connections and global parallelisms}
Let us now discuss an example of sheaf of vector field that does not fit in the theory of $G$-fields described in Section~\ref{sub:Gstructuresinfsymmetries}, but that can nevertheless be described in terms of a $G$-structure of finite type.

Let $M$ be an $n$-dimensional manifold endowed with an affine connection $\nabla$.
Let $\mathrm{FR}(TM)$ denote the frame bundle of $M$, which is the total space of a $\mathrm{GL}(n)$-principal bundle on $M$, and it has dimension equal to $n^2+n$. Denote by $\pi:\mathrm{FR}(TM)\to M$ the canonical projection.
Let us recall how to define a global parallelism on $\mathrm{FR}(TM)$, i.e., a $\{1\}$-structure, using the connection
$\nabla$.

Given $p\in\mathrm{FR}(TM)$, i.e., $p:\mathds R^n\to T_xM$ is an isomorphism for some $x\in M$,
then $T_p\big(\mathrm{FR}(TM)\big)$ admits a splitting into a vertical and a horizontal space, denoted respectively
$\mathrm{Ver}_p$ and $\mathrm{Hor}_p$. The vertical space $\mathrm{Ver}_p$ is canonical, and canonically identified with $\mathrm{Lin}(\mathds R^n,T_xM)$. The horizontal space $\mathrm{Hor}_p$ is defined by the connection $\nabla$, and it is isomorphic to $T_xM$ via the differential  $\mathrm d\pi(p):\mathrm{Hor}_p\to T_xM$.

Fix a frame $q:\mathds R^{n^2}\to\mathrm{Lin}(\mathds R^n,\mathds R^n)$.
A frame for $\mathrm{Ver}_p$ is given by $p \circ q:\mathds R^{n^2}\to\mathrm{Ver}_p$.
A frame for $\mathrm{Hor}_p$ is given by
\[\left(\mathrm d\pi_p\big\vert_{\mathrm{Hor}_p}\right)^{-1}\circ p:\mathds R^n\longrightarrow\mathrm{Hor}_p.\]
Putting these two things together, we obtain a global frame of $\mathrm{FR}(TM)$, which is determined by the choice of the horizontal distribution.

Clearly, not every global frame of $\mathrm{FR}(TM)$ arises from a connection on $M$. On the other hand, if two connections define the same global frame, then they coincide. Moreover, a smooth diffeomorphism $f:M\to M$ is affine, i.e., it preserves $\nabla$, if and only if the induced map $\widetilde f:\mathrm{FR}(TM)\to\mathrm{FR}(TM)$ preserves the associated global frame. Similarly, a (local) field $X$ on $M$ is an infinitesimal affine symmetry, i.e., its flow consists of affine local diffeomorphisms, if and only if its lift $\widetilde X$ to a local field on $\mathrm{FR}(TM)$ (see Section~\ref{sub:lifting}) is an infinitesimal symmetry of the $\{1\}$-structure on $\mathrm{FR}(TM)$ associated with the connection $\nabla$.
\smallskip

Let us denote by $\mathcal F^\nabla$ the sheaf of local affine fields of the affine manifold $(M,\nabla)$.
From the above considerations, we obtain easily:
\begin{prop}\label{thm:affinefieldsuniquebounded}
Let $(M,\nabla)$ be an affine manifold. The sheaf $\mathcal F^\nabla$ is admissible.
\end{prop}
\begin{proof}
The map $X\mapsto\widetilde X$ gives an injection of $\mathcal F^\nabla$ into the sheaf of local fields whose flow preserves the global frame on $\mathrm{FR}(TM)$ associated with  $\nabla$. The conclusion follows easily from the fact that both the bounded rank and the unique continuation property are inherited by sub-sheaves (Remark~\ref{rem:admissiblesubsheaf}).
\end{proof}
We will show in Section~\ref{sec:sprays} that the same result holds in the more general case of the sheaf of  vector  fields whose flow preserves the geodesics of a spray.
%
%

\section{On the regularity condition}
\label{sec:regularity}
The central assumption of Theorem~\ref{thm:globalregular}, i.e., that the space of germs of the admissible sheaf $\mathcal F$ should have constant dimension along $M$, is in general hard to establish. We will present here two situations where such regularity assumption is satisfied, namely the \emph{real-analytic} and the \emph{transitive} case.

\subsection{The real-analytic case}
\label{sec:realanalyticcase}
Let us now consider a real-analytic manifold $M$, and let $\mathfrak X^\omega$ denote the sheaf of local real-analytic vector fields on $M$. A sheaf $\mathcal F$ of local vector fields on $M$ is said to be \emph{real-analytic} if $\mathcal F(U)\subset\mathfrak X^\omega(U)$ for all open subset $U\subset M$.

\begin{prop}\label{thm:realanalyticsheaf}
If $\mathcal P\subset\mathrm{FR}(TM)$ is a real-analytic $G$-structure of finite type on $M$,
then the sheaf $\mathcal F^\mathcal P$ of local $G$-fields on $M$ (Definition~\ref{thm:defsheafGfields}) is real-analytic.
\end{prop}
\begin{proof}
This is proven by induction on the order of the $G$-structure.
Recall that a local field $X$ is a $G$-field if and only if admits a lifting $\widetilde X$ to $\mathcal P$ that preserves the first prolongation $\mathcal P^{(1)}$ of $\mathcal P$. By the induction hypothesis, $\widetilde X$ is real-analytic, and therefore so is $X$, which is the projection of $\widetilde X$
by the real-analytic map $\mathrm d\pi:T\mathcal P\to TM$.

It remains to show the basis of induction, i.e., the case $G=\{1\}$. In this case, the $G$-structure is a global frame $(X_1,\ldots,X_n)$ of real-analytic fields. Let us define:
\begin{equation}\label{eq:localdiffeo1struct}
f(t_1,\ldots,t_n,x)=F_{t_1}^{X_1}\circ\cdot\circ F_{t_n}^{X_n}(x),
\end{equation}
where $F_t^{X_i}$ is the flow of the field $X_i$ at time $t$. The map $f$ is real-analytic. For $x\in M$ fixed, the map $(t_1,\ldots,t_n)\mapsto f^x(t_1,\ldots,t_n):=f(t_1,\ldots,t_n,x)$ is a real-analytic diffeomorphism from a neighborhood of $0$ in $\mathds R^n$ to a neighborhood of $x$ in $M$.
If $X$ is a local field that preserves the $\{1\}$-structure, then:
\begin{equation}\label{eq:Xpreserving1}
X\big(f(t_1,\ldots,t_n,x)\big)=\frac{\partial f}{\partial x}(t_1,\ldots,t_n,x)X(x).
\end{equation}
If we fix $x$, equality \eqref{eq:Xpreserving1} says that $X$ is real-analytic in a neighborhood of $x$. This concludes the proof.
\end{proof}
\begin{rem}\label{thm:rem1fielddetermined}
The last argument in the proof of Proposition~\ref{thm:realanalyticsheaf} shows that, given a frame $X_1,\ldots,X_n$ on a manifold $M$, if $X$ is a local field whose flow preserves the frame, defined in a neighborhood of a point $q\in M$, then for $z$ near $q$, $X(z)$ is given by $\frac{\partial f}{\partial x}\big((f^q)^{-1}(z),q\big)X(q)$, where $f$ is given in \eqref{eq:localdiffeo1struct} and $f^q(t_1,\ldots,t_n):=f(t_1,\ldots,t_n,q)$.
\end{rem}
\begin{rem}\label{thm:notfintypenotal}
We remark that Proposition~\ref{thm:realanalyticsheaf} fails to hold for $G$-structures that are
not of finite type. For instance, given any real-analytic manifold $M^n$, one can consider the real-analytic $\mathrm{GL}(n)$-structure on $M$ given by $\mathcal P=\mathrm{FR}(TM)$. Then, any local vector field is a $G$-field. See also Remark~\ref{rem:locuniqvsanalyt}.
\end{rem}

Let us now introduce the following class of real-analytic sheaves. For a given open subset $U\subset M$, we will denote by $C^\omega(U)$ the space of
real-analytic functions $f:U\to\mathds R$.
\begin{defi}\label{def:realanaldet}
Let $\mathcal F$ be a real-analytic sheaf. We say that $\mathcal F$ is \emph{analytically determined} if for all open subset $U$ there exists a collection $\mathcal A(U)$ of functions $F:\mathfrak X^\omega(U)\to C^\omega(U)$ satisfying the following two properties:
\begin{itemize}
\item[(a)] Locality: given open sets $V\subset U$ and $F\in\mathcal A(U)$, then for all $X\in\mathfrak X^\omega(U)$, the restriction
$F^X\big\vert_V$ is equal to $G^{X\vert_V}$ for some $G\in\mathcal A(V)$.
\item[(b)] For all open subset $U$ and all $X\in\mathfrak X^\omega(U)$,  $X\in\mathcal F(U)$ if and only if $F^X\equiv0$ on $U$ for all $F\in\mathcal A(U)$.
\end{itemize}
\end{defi}
\begin{rem}\label{rem:refinaranalyticdeter}
The condition of being analytically determined can be relaxed considering, not all the open sets $U$, but a base of sufficiently small open sets.
\end{rem}
The definition above is modeled on the following example.
\begin{exe}\label{exa:Killinganaldet}
If $(M,g)$ is a real-analytic Riemannian manifold, and $\mathcal F$ is the sheaf of local Killing fields of $(M,g)$, then $\mathcal F$ is analytically determined. Given an open subset $U\subset M$, then the collection $\mathcal A(U)$ is the set of functions $\big\{F_V:V\in\mathfrak X^\omega(U)\big\}$, defined by:
\[F_V^X=g(\nabla_VX,V).\]
\end{exe}
The property of an analytically determined sheaf $\mathcal F$ which is important for us is the fact that real-analytic extensions of local
$\mathcal F$-fields are $\mathcal F$-fields:
\begin{lemma}\label{thm:anextF}
Let $\mathcal F$ be an analytically determined sheaf, $U\subset M$ a connected open set and $V\subset U$ a non-empty open subset. If $X\in\mathfrak X^\omega(U)$ is such that $X\big\vert_V$ belongs to $\mathcal F(V)$, then $X\in\mathcal F(U)$.
\end{lemma}
\begin{proof}
Given an arbitrary $F\in\mathcal A(U)$, then, since $X\big\vert_V\in\mathcal F(V)$, by properties (a) and (b) in Definition~\ref{def:realanaldet}, $F^X\big\vert_V=0$.
Since $U$ is connected, $V$ non-empty and $F^X$ real-analytic, it follows $F^X\equiv0$ in $U$. Thus, by (b) of Definition~\ref{def:realanaldet},
$X\in\mathcal F(U)$.
\end{proof}
Let us generalize Example~\ref{exa:Killinganaldet} to arbitrary $G$-structures of finite type.
\begin{prop}\label{thm:analyticallydetsheaf}
If $\mathcal P\subset\mathrm{FR}(TM)$ is a real-analytic $G$-structure of finite type on $M$,
then the sheaf $\mathcal F^\mathcal P$ of local $G$-fields on $M$ is analytically determined.
\end{prop}
\begin{proof}
Let us choose a compatible real-analytic connection $\nabla$ with torsion $T$. The existence of real-analytic connections compatible with real-analytic $G$-structures is proved in Appendix~\ref{app:RAcompconnections}.
For any $U\subset M$ and any $X\in\mathcal F^\mathcal P(U)$, the $(1,1)$-tensor field $\eta_X:=\nabla X+T(X,\cdot)$ is real-analytic on $U$.
Consider the real-analytic vector bundle $\mathfrak g_U$ on $U$, whose fiber at the point $p\in U$ is the Lie algebra $\mathfrak g_p$ (see Subsection~\ref{sub:charGfields}). This is a vector sub-bundle of the vector bundle $\mathrm{End}(TU)$ whose fiber at $p\in U$ is given by the endomorphisms of $T_p U$.
Consider the dual vector bundle $\mathrm{End}(TU)^*$, and its sub-bundle $\mathfrak g_U^o$, whose fiber at the point $p\in U$ is the annihilator $\mathfrak g_p^o$ of $\mathfrak g_p$. All these bundles are obtained by natural operations, and they are real-analytic.

Having settled these notations, we can now describe the collection $\mathcal A(U)$ as follows:
\[\mathcal A(U)=\Big\{F_\varrho:\varrho\ \text{real-analytic section of $\mathfrak g_U^o$}\Big\},\]
where, for $X\in\mathfrak X^\omega(U)$, $F_\varrho^X:U\to\mathds R$ is defined by:
\[F_\varrho^X=\varrho(\eta_X).\]
Property (a) of Definition~\ref{def:realanaldet} for the family $\mathcal A(U)$ is immediate. Now, observe that $F_\varrho^X\equiv0$ on $U$ for all $\varrho$
real-analytic section of $\mathfrak g_U^o$ if and only if $\eta_X(p)\in\mathfrak g_p$ for all $p\in U$.
By Proposition~\ref{thm:charGfields}, this implies that $\mathcal A(U)$ satisfies property (b) of Definition~\ref{def:realanaldet}, and proves that
$\mathcal F^\mathcal P$ is analytically determined.
\end{proof}

Let us now show how analytic determinacy of a sheaf $\mathcal F$ is related to the property of regularity.
\begin{prop}\label{thm:ADimpliesreg}
Let $M$ be a real-analytic connected manifold, and let $\mathcal F$ be an admissible sheaf of vector fields which is analytically determined.
Assume that every $p\in M$ admits an open neighborhood $\mathcal V_p$ with the following property: given any $q\in\mathcal V_p$ and any $\mathcal F$-germ $\mathfrak g\in\mathfrak G_q^\mathcal F$, then there exists $X\in\mathfrak X^\omega(\mathcal V_p)$ such that $\mathrm{germ}_q(X)=\mathfrak g$.
Then $\mathcal F$ is regular.
\end{prop}
\begin{proof}
Given $p\in M$, $\mathcal V_p\subset M$, $q\in\mathcal V_p$, $\mathfrak g\in\mathfrak G^\mathcal F_q$, $X\in\mathfrak X^\omega(\mathcal V_p)$ and  $\mathrm{germ}_q(X)=\mathfrak g$ as in the statement, since $\mathcal F$ is analytically determined, by Lemma \ref{thm:anextF}, it follows that $X\in\mathcal F(\mathcal V_p)$. In the terminology of Section~\ref{sub:Fspecial}, $\mathcal V_p$ is $\mathcal F$-special for all its points, and
therefore $\kappa^\mathcal F$ is constant on $\mathcal V_p$, i.e., the map $\kappa^\mathcal F:M\to\mathds N$ is locally constant. Since $M$ is connected, then $\kappa^\mathcal F$ is constant and $\mathcal F$ is regular.
\end{proof}
 Proposition~\ref{thm:ADimpliesreg}  can be applied in very general situations, when the values along a smooth curve of a field belonging to the sheaf $\mathcal F$ are given by ordinary differential equations with real-analytic coefficients. In this case, local real-analytic extensions of $\mathcal F$-germs are obtained by integrating along a real-analytic family of curves whose images cover a neighborhood of a given point (for instance, radial geodesics issuing from the point). An example where Proposition~\ref{thm:ADimpliesreg} applies is given, again, by $G$-structures of finite type. For such a result, we require first the following technical lemma

 \begin{lemma}\label{lemmaprojectable}
Consider $\pi:C\rightarrow M$ an analytic submersion and assume that there exists an analytic vector field $\tilde{X}$ defined in an open connected set $U\subset C$. If $\tilde{X}$ is projectable in an open subset $V\subset U$, then it is projectable on the entire $U$.
 \end{lemma}
\begin{proof}
 Our aim is to show that the sheaf of projectable fields on $C$ is analytically determined, and so, the result follows from Lemma \ref{thm:anextF}. For this, recall that  $\tilde{X}$ is projectable if and only if $\pi_*([\tilde{X},V])=0$ for every vertical vector field $V$. Then, for any open set $U$ admitting analytic adapted coordinates $(x^1,\dots,x^n,y^1,\dots,y^m,U)$ for $C$, where $\{\partial_{y^i}\}$ is a basis of the vertical space, consider the collection $\mathcal{A}(U)$  of functions $\{F_i:i=1,\dots,m\}$ defined by
\[
F^{\tilde{X}}_i=\sum_{j=1}^n g(\pi_*([\tilde{X},\partial_{y^i}]),\partial_{x^j})^2,
\]
where $g$ is an auxiliary Riemannian metric on $M$.
%
%
\end{proof}
\begin{thm}\label{thm:FPanalyticregular}
Let $M^n$ be a connected real-analytic manifold, $G\subset\mathrm{GL}(n)$ a Lie subgroup, and $\mathcal P\subset\mathrm{FR}(TM)$ a real-analytic $G$-structure on $M$. Then, the sheaf $\mathcal F^\mathcal P$ is regular.
\end{thm}
\begin{proof}
By Proposition~\ref{thm:analyticallydetsheaf}, $\mathcal F^\mathcal P$ is analytically determined. We will show that $\mathcal F^\mathcal P$ satisfies
the assumptions of Proposition~\ref{thm:ADimpliesreg}. Using induction on the order of the $G$-structure, it is easy to see that it suffices to consider the case of a $\{1\}$-structure. For the induction step, let $x\in M$ be fixed, choose $p\in\pi^{-1}(x)\in\mathcal P$, and let $\mathcal P_0$ be the connected component of $\mathcal P$ that contains $p$.
Then, $\mathcal P_0$ has a real-analytic $G^{(1)}$-structure $\mathcal P^{(1)}_0$, which by induction hypothesis satisfies the assumptions of Proposition~\ref{thm:ADimpliesreg}. Denote by $\mathcal U^{(1)}_p$ an open neighborhood that satisfies the real-analytic extension property, as described in Proposition~\ref{thm:ADimpliesreg}, for the sheaf $\mathcal F^{\mathcal P_0^{(1)}}$ on $\mathcal P_0$. Then,
the projection $\mathcal U_x=\pi\big(\mathcal U^{(1)}_p\big)$ is an open neighborhood of $x$, and it satisfies the same real-analytic extension property for the sheaf $\mathcal F^\mathcal P$ on $M$. In fact, given an analytic $G$-field $X$ defined in $\mathcal V\subset \mathcal U_x$, take $\tilde{X}$ the lifting of $X$ on an open subset $\mathcal V^1\subset \mathcal U_p^1$. From construction, $\tilde{X}$ is extensible analytically on the entire $U_p^1$ and, from previous lemma, the extended field is projectable on $\mathcal U_x$.

\smallskip

Let us prove the result for $\{1\}$-structures.
Let  $(X_1,\ldots,X_{n})$ be a real-analytic global frame of $M$.
Consider the map $f:\mathds R^n\times M\to M$ defined in \eqref{eq:localdiffeo1struct}; for $x\in M$, set $f^x(t_1,\ldots,t_n)=f(t_1,\ldots,t_n,x)$. Observe that as the map $\tilde{f}:\R^n\times M\to M\times M$ defined as $\tilde{f}(t_1,\ldots,t_n,x)=(f(t_1,\ldots,t_n,x),x)$ is a local diffeomorphism around $(0,x)$, for every $p\in M$, we can choose a (connected) neighborhood ${\mathcal V}_p$ of $p$ such that
the map \[(t_1,\ldots,t_n)\longmapsto f^q(t_1,\ldots,t_n)\] gives a real-analytic diffeomorphism from a neighborhood $\mathcal W_{q}$  of $0$ in $\mathds R^n$ to ${\mathcal V}_p$.

 Let $q\in\mathcal V_p$ be fixed, and let $X$ be any local field, defined in some connected neighborhood $\mathcal O\subset\mathcal V_p$ of $q$, whose flow preserves the frame $(X_1,\ldots,X_n)$. Set $v=X(q)$. Then, the field $\widetilde X$ defined by:
\[\widetilde X(z)=\frac{\partial f}{\partial x}\big((f^q)^{-1}(z),q\big)v,\] for $z\in {\mathcal V}_p$, is a real-analytic extension of $X$ to $\mathcal V_p$, see Remark~\ref{thm:rem1fielddetermined}. This shows that when $\mathcal P$ is a real-analytic $\{1\}$-structure, then $\mathcal F^\mathcal P$ satisfies the assumptions of Proposition~\ref{thm:ADimpliesreg}.
\end{proof}
\subsection{The transitive case}
Given a sheaf of vector fields $\mathcal F$ on the manifold $M$, let $\mathcal D(\mathcal F)$ denote the associated pseudo-group of local diffeomorphisms of $M$. Recall that this is the pseudo-group generated by the family of all local diffeomorphisms given by the local flow of elements of $\mathcal F$.
\begin{defi}\label{thm:D(F)transitive}
The sheaf $\mathcal F$ is said to be \emph{full} if for every $X\in\mathcal F$ and any $\varphi\in\mathcal D(\mathcal F)$, the pull-back $\varphi^*(X)$ belongs to $\mathcal F$. Given a complete sheaf $\mathcal F$,
the pseudo-group $\mathcal D(\mathcal F)$  is said to be \emph{transitive} if it acts transitively on $M$, i.e., if given any two points $x,y\in M$, there exists $\varphi\in\mathcal D(\mathcal F)$ such that $\varphi(x)=y$.
\end{defi}
As an immediate consequence of the definition, we have the following:
\begin{prop}\label{thm:transitiveimpliesregular}
Let $\mathcal F$ be a full sheaf of vector fields on $M$ which has bounded rank. If $\mathcal D(\mathcal F)$ is transitive, then $\mathcal F$ is regular.
\end{prop}
\begin{proof}
Given $x,y\in M$, and $\varphi\in\mathcal D(\mathcal F)$ such that $\varphi(x)=y$, then the pull-back $\varphi^*$ induces an isomorphism between the space of germs $\mathfrak G_y^\mathcal F$ and $\mathfrak G_x^\mathcal F$.
\end{proof}
An interesting consequence of fullness is the following result:
\begin{prop}\label{thm:Liealgebrastr}
Let $\mathcal F$ be a full sheaf of smooth local vector fields on a manifold $M$. Assume that $\mathcal F$ is closed by $C^\infty$-convergence on compact sets. Then, $\mathcal F$ is closed by Lie brackets, i.e., $\mathcal F(U)$ is a Lie algebra for every open set $U\subset M$. This is the case, in particular, for the sheaf of vector fields whose flow preserves a $G$-structure on $M$.
\end{prop}
\begin{proof}
See \cite[Remark (b), p.\ 10]{SinSte65}.
\end{proof}
The sheaf $\mathcal F$ is called \emph{transitive} if the evaluation map $\mathrm{ev}_x:\mathfrak G_x^\mathcal F\to T_xM$ is surjective for all $x$. In other words, $\mathcal F$ is transitive if for all $x\in M$ and all $v\in T_xM$, there exists a local $\mathcal F$-field $K$ defined around $x$ such that $K_x=v$.
When $\mathcal F$ is a Lie algebra sheaf, the transitivity of  $\mathcal F$ implies the transitivity of $\mathcal D(\mathcal F)$.
\begin{prop}\label{thm:transitivities}
If $\mathcal F$ is a full Lie algebra sheaf of local fields on $M$ which is transitive, then $\mathcal D(\mathcal F)$ is transitive.
\end{prop}
\begin{proof}
See \cite[Proposition~1.1 and \S~1.3]{SinSte65}.
\end{proof}
Proposition~\ref{thm:transitivities} applies in particular to the sheaf of local vector fields that are infinitesimal symmetries of a $G$-structure, see Section~\ref{sub:Gstructuresinfsymmetries}.
\section{Killing and conformal fields for Finsler manifolds}
\label{sec:KillingFinsler}
Let $(M,F)$ be a Finsler manifold with $F:TM \rightarrow [0,+\infty)$, namely, a continuous positive homogeneous function which is smooth away from the zero section and such that the fundamental tensor defined as
\[g_v(u,w)=\frac{1}{2}\frac{\partial^2}{\partial t\partial s}\Big|_{t=s=0}F(v+tu+sw)^2\]
for every $v\in TM\setminus 0$ and $u,w\in T_{\pi(v)}M$, is positive definite. Let us recall that the Chern connection associated with a Finsler metric can be seen as a family of affine connections, namely, for every non-null vector field $V$ on an arbitrary open subset $U\subset M$, we can define an affine connection $\nabla^V$ in $U$, which is torsion-free and almost $g$-compatible (see for example \cite{Jav13}). The Chern connection also determines a covariant derivative along any curve $\gamma:[a,b]\rightarrow M$, which depends on a non-null reference vector $W$ and which we will denote as $D^W_\gamma$.
Let us define a local Killing (resp. conformal) vector field of $(M,F)$  as a vector field $K$ such that the flow of $K$ gives local isometries (resp. conformal maps) for every $t\in\R$ where it is defined. Furthermore, following \cite{JaLiPi11} we say that a map $\varphi:(M,F)\rightarrow (M,F)$ is an almost isometry if there exists $f:M\rightarrow \R$ such that $\varphi^*(F)=F-df$. This kind of maps allows one to characterize the conformal maps of a stationary spacetime in terms of its Fermat metric as shown in \cite{JaLiPi11}. We will say that a vector field is {\it almost Killing} if its flow gives local almost isometries.
 Let us begin by giving a characterization of (local) Killing, almost Killing and conformal vector fields.

\begin{prop}\label{thm:diffeqnFinslerKilling}
Let $(M,F)$ be a Finsler manifold and $K$ a vector field in an open subset $U\subset M$.
\begin{itemize}
\item[(i)] $K$ is Killing if and only if $g_v(v,\nabla^v_vK)=0$ for every $v\in \pi^{-1}(U) \setminus 0$, where $\pi:TM\rightarrow M$ is the canonical projection.\smallskip
\item[(ii)] $K$ is conformal if and only if there exists a function $f:U\rightarrow (0,+\infty)$ such that $g_v(v,\nabla^v_vK)=f(\pi(v)) F(v)^2$ for every $v\in \pi^{-1}(U) \setminus 0$.
\item[(iii)]  $K$ is almost Killing if and only if there exists a one-form $\xi$ such that $g_v(v,\nabla^v_vK)=F(v)\xi(v)$ for every $v\in \pi^{-1}(U) \setminus 0$ and the one-form $\tilde{\xi}_s$, defined as
\begin{equation}\label{intone-form}
\tilde{\xi}_s(v)=\int_0^s \xi(\varphi^*_{\mu}(v))d\mu,
\end{equation}
 where $\varphi_{\mu}$ is the flow of $K$, is closed for every $s\in \R$ such that the flow of $K$ is defined in $[0,s]$.
\end{itemize}
\end{prop}
\begin{proof}
Given $v\in \pi^{-1}(U)$, let $\gamma:[-1,1]\rightarrow U$ be a curve such that $\dot\gamma(0)=v$. Now let $\phi_s$ be the flow of $K$ in the instant $s\in \R$ for $s$ small enough in such a way that the flow is defined for that instant along $\gamma$. Define the two-parameter $\Lambda(t,s)=\phi_s(\gamma(t))$ and denote by $\beta_t(s)=\Lambda(t,s)=\gamma_s(t)$. By definition $\dot\beta_t(s)=K$. We have that
\begin{equation}\label{dFdois}
\frac{d}{ds} F(\dot\gamma_s(t))^2=2g_{\dot\gamma_s}(D^{\dot\gamma_s}_{\beta_t}\dot\gamma_s,\dot\gamma_s)=2g_{\dot\gamma_s}(D^{\dot\gamma_s}_{\gamma_s}K,\dot\gamma_s).
\end{equation}
Here we have used \cite[Proposition 3.2]{Jav13}.
Observe that $K$ is Killing if and only if
$F(\dot\gamma(t))=F(\dot\gamma_s(t))$ for every curve $\gamma$. By the above equation this is equivalent to $g_{v}(\nabla^{v}_{v}K,v)=0$ for every $v\in \pi^{-1}(U)$ and the proof of (i) follows. When $K$ is conformal, we have that $F(\dot\gamma_s(t))^2=h(s,\gamma(t))F(\dot\gamma(t))^2$ and then
\[ g_v(\nabla^v_vK,v)=f (\pi(v)) F(v)^2\]
where $f(p)=\frac 12\frac{\partial h}{\partial s}(0,p)$. Moreover, if $K$ satisfies an equation of this type for some function $f$, then if we define $h(s,p)= 2\int_0^s f(\phi_\mu(p)) d\mu$, we deduce that the flow $\phi_s$ is conformal with
$F(\phi^*_s(v))^2=h(s,\pi(v)) F(v)^2$. This concludes the proof of part (ii). For the last part, observe that if $K$ is an almost Killing field, then for every $s$ such that the flow of $K$, $\varphi_s$, is (locally) defined, there exists a function $f_s:M\rightarrow \R$ such that $F(\dot\gamma_s(t))=F(\dot\gamma(t))+d f_s(\dot\gamma(t))$. Deriving both sides of last equality with respect to $s$, using \eqref{dFdois} and evaluating in $s=0$ we get
\[g_v(v,\nabla^v_vK)=F(v)\xi(v)\]
where $\xi(v)=\frac{\partial}{\partial s} df_s(v)|_{s=0}$. Furthermore, taking $v=\dot\gamma_s(t)$ as above, we have
\[\frac{1}{F(\dot\gamma_s(t))}g_{\dot\gamma_s}(D^{\dot\gamma_s(t)}_{\gamma_s}K,\dot\gamma_s(t))=\xi(\dot\gamma_s(t)).\]
Integrating with respect to $s$ we get that
\[F(\dot\gamma_s(t))-F(\dot\gamma(t))=\int_0^s \xi(\dot\gamma_{\mu}(t))d\mu\]
or equivalently $F(\varphi_{s}^*(v))-F(v)=\int_0^s \xi(\varphi_{\mu}^*(v))d\mu=\tilde{\xi}_s(v)$. This implies that  $\tilde{\xi}_s$ is closed because $\varphi_s$ is an almost isometry. The converse follows easily from the above equations.

\end{proof}

\begin{rem}\label{rem:analiticadeterminadofinsler}
When $(M,F)$ has analytic regularity, the previous proposition ensures that  Killing, almost Killing and conformal fields are analytically determined. In fact, the Killing case follows as in Example \ref{exa:Killinganaldet}. The conformal case however is a little trickier as it requires the definition of the function $f$. For this, let us consider $U$ a small enough open set (recall Remark \ref{rem:refinaranalyticdeter}) so that there exists an analytic vector field  $V_0\in \mathfrak X^\omega(U)$ without zeroes. Then, define the analytic function
\[f(\pi(V_0))=\frac{g_{V_0}(V_0,\nabla_{V_0}^{V_0}K)}{F(V_0)^2}\]
and define the collection $\mathcal{A}(U)$ as the set of functions $\{H_V:V\in\mathfrak X^\omega(U)\}$ defined by
\[
H^X_V=g_{V}(V,\nabla_V^V X)-f(\pi(V))F(V)^2
\]
It follows that $K$ is conformal if, and only if, $H^K\equiv 0$ for all $H\in\mathcal{A}(U)$. Finally, in the case of almost isometries we can proceed in a similar way. Choose a small enough open set $U$ in such a way that it admits an analytic frame $(V_1,V_2,\ldots,V_n)$. Then given a vector field $X$ in $U$ define the one-form $\xi$ in $U$ such that $\xi(V_i)=\frac{1}{F(V_i)} g_{V_i}(V_i,\nabla_{V_i}^{V_i}X,V_i)$ for $i=1,\ldots,n$. Moreover, define $\tilde{\xi}_s$ as in \eqref{intone-form} using the flow of $X$. Then given $V,W \in\mathfrak X^\omega(U)$ and  $s\in\R$ smal enough in such a way that $\tilde{\xi}_s$ is defined in $U$,  choose the collection $\mathcal{A}(U)$ as the set of functions
\[
H^X_{V,W,s}=(g_{V}(V,\nabla_V^V X)-F(V)\xi(V))^2+(d\tilde{\xi}_s(V,W))^2.
\]

\end{rem}

In order to see that local Killing and conformal vector fields of Finsler manifolds constitute and admissible sheaf, we need to introduce an average Riemannian metric associated with a Finsler metric. There are several notions of  Riemannian metric associated with a Finsler structure (see \cite{ricardo,MRTZ09}); we will consider the following one.

For $p\in M$, let us set $B_p=\big\{v\in T_pM: F(v)\leq 1\big\}$, $\mathcal S_p=\partial B_p$, and let $\Omega_p$ be the unique $n$-form
on $T_pM$ such that $B_p$ has volume equal to $1$. Then we define the $(n-1)$-form $\omega_p$ on $\mathcal S_p$ at $u\in\mathcal S_p$ as
\[\omega_p(\eta_1,\eta_2,\ldots,\eta_{n-1})=\Omega_p(\eta_1,\eta_2,\ldots,\eta_{n-1},u)\]
for every $\eta_1,\eta_2,\ldots,\eta_{n-1}\in T_uS_p$. The average Riemannian metric $g_R$ at the point $p\in M$ is defined as
\begin{equation}\label{eq:average}
g_{R}(v,w):=\int_{u\in \mathcal{S}_p} g_{u}(v,w)\omega_p.
\end{equation}
The first observation about the average Riemannian metric $g_R$ is that if $\phi:U\subset M\rightarrow  V\subset M$ is a local isometry (resp., a local conformal map) of $(M,F)$, then it is also a local isometry (resp., a local conformal map) of $(M,g_R)$; in particular, if $K$ is a local Killing (resp., conformal) vector field of $(M,F)$, then it is also a local Killing (resp., conformal) vector field of $(M,g_R)$. This implies that the sheaf of local Killing (resp. conformal) vector fields of a Finsler manifold is admissible, so that Theorem \ref{thm:globalregular} applies whenever $M$ is simply-connected and the sheaf of local Killing (resp. conformal) vector fields is regular. This is the case, for example when the manifold $(M,F)$ is real-analytic. In the case
of almost Killing fields, we will need to use the symmetrization of a Finsler metric defined as $\hat{F}(v)=\frac{1}{2}(F(v)+F(-v))$. Recall that in \cite[Corollaries 4.3 and 4.6]{JaSa11}, it was proved that $\hat{F}$ is a (strongly convex) Finsler metric and in \cite{JaLiPi11} that an almost isometry of $F$ is an isometry of $\hat{F}$, which easily implies that almost Killing fields are admissible.
\begin{prop}\label{prop:analyFinsCase}
Let $(M,F)$ be a simply-connected real-analytic Finslerian manifold. Then, every local Killing (resp. almost Killing, conformal) vector field  can be extended to a unique Killing (resp. almost Killing, conformal) vector field defined on the whole manifold $M$.
\end{prop}
\begin{proof}
As we have seen above, the sheaf of local Killing (resp. conformal) vector fields of a Finsler manifold is admissible and analytically determined. Therefore, we only have to show that any local Killing (resp. conformal) field is extensible analytically to a prescribed open set and apply Proposition \ref{thm:ADimpliesreg}. Let $\tilde{X}$ be a local Killing (resp. conformal) vector field of $(M,F)$ defined on an open subset $U$. As $\tilde{X}$ is also local Killing (resp. conformal) vector field of the average Riemannian metric given in \eqref{eq:average}, and $(M,g_R)$ is a simply-connected real-analytic Riemannian manifold, \cite[Theorems 1 \& 2]{nomizu60} (resp. \cite[Lemma 3]{LedOba70}) ensures that $\tilde{X}$ can be extended to an analytic Killing (resp. conformal) vector field $X$ for $g_R$ defined on the entire $M$. In particular, $\tilde{X}$ admits a global analytic extension and Proposition \ref{thm:ADimpliesreg} applies. A similar reasoning can be done for almost Killing fields using that they are Killing fields of the symmetric Finsler metric $\hat{F}$.

\end{proof}
\subsection{An application}
Using the above results, we are able to give a nice characterization of the homogeneity of Finsler manifolds. For this, we first need the following technical lemma, which is well known on Semi-Riemannian Geometry

\begin{lemma}
On a (forward) complete Finslerian manifold $M$ every global Killing vector field $K$ is complete.
\end{lemma}

The proof follows by using the same arguments as in the semi-Riemannian case (see \cite[Proposition 9.30]{ON83} for instance), and taking into account that the Jacobi fields are well-enough behaved in the Finslerian settings (see \cite[Section 3.4]{Java}, especially Proposition 3.13 and Lemma 3.14). With this previous result at hand, we are ready to prove the following characterization:

\begin{thm}
Let $(M,F)$ be a (connected and) simply-connected, (forward) complete Finslerian manifold. Then, $M$ is homogeneous if, and only if, the following two conditions hold:
\begin{itemize}
\item[(i)] $\kappa^{\mathcal F}_p$ is constant on $M$ ;\smallskip

\item[(ii)] for some point $p_0\in M$ and all $v\in T_{p_0}M$ there exists a local Killing vector field $V$ such that $V(p_0)=v$.
\end{itemize}
\end{thm}

\begin{proof}

Let us show that homogeneity implies (i) and (ii). By homogeneity, the pseudo-group ${\mathcal D}(\mathcal{F})$ associated with the full sheaf ${\mathcal F}$ of local Killing vector fields is transitive, and so, Proposition \ref{thm:transitiveimpliesregular} ensures (i). For (ii) recall that, for each $v\in T_{p_0}M$ we can consider $\epsilon>0$ and some differentiable 1-parameter family of isometries $\{\psi_t\}_{t\in (-\epsilon,\epsilon)}:M\rightarrow M$ with $\psi_0(p_0)=p_0$ and $\frac{\partial}{\partial t}\big|_{t=0}(\psi_{t}(p_0))=v$. Then, $V(p)=\frac{\partial}{\partial t}|_{t=0}(\psi_{t}(p))$ is a Killing vector field with $V(p_0)=v$, and so, (ii) is satisfied.

Let us now assume that (i) and (ii) hold. As $\kappa^{\mathcal F}_p$ is constant, the sheaf of local Killing vector fields ${\mathcal F}$, which is an admissible sheaf, is also a regular one. Then, as $M$ is simply-connected, Theorem \ref{thm:globalregular} ensures that any local Killing vector field extends globally. Moreover, from previous Lemma, such a global Killing vector fields are complete, and so, the pseudo-group ${\mathcal D}({\mathcal F})$ is truly a group. Let $G:={\mathcal D}({\mathcal F})$ and $H$ the isotropy group of the action of $G$ in $M$ at $p_0$. Then we have a map $i:G/H\rightarrow M$, given by $g\mapsto g(p_0)$, which is a local diffeomorphism. This follows from hypothesis (ii) for $p_0$ at $[e]\in G/H$, but it translates by the action of $G$ to all the points in $G/H$.  Observe that the map $i$ is also injective. Moreover, if you endow $G/H$ with the metric $i^*(F)$, then as $F$ is invariant by the action of $G$, we obtain a homogeneous Finslerian  manifold $(G/H,i^*(F))$, which is then complete (repeat the same proof as in \cite[Remark 9.37]{ON83}). As a consequence $i(G/H)=M$ and $M$ is homogeneous.

%

\end{proof}

\section{Pseudo-Finsler local Killing fields}
\label{sec:pseudoFinsler}
Let us now consider infinitesimal symmetries of another type of structure that cannot be described in terms of a finite order $G$-structure: pseudo-Finsler structure. These structures are the indefinite counterpart to Finsler structures, in the same way Lorentzian or pseudo-Riemannian metrics are the indefinite counterpart to Riemannian metrics. A precise definition will be given below.
Let us observe here that, unlike the Finsler (i.e., positive definite) case, the construction of average metrics is not possible in this situation, and we have to resort to a different type of construction, based on the notion of Sasaki metrics.
This construction also works when the fundamental tensor is positive definite but the metric is conic in the sense that it is only defined in some directions \cite{JaSa11}. In such a case it is not possible to construct an average Riemannian metric. A very classical example of this situation  is found in  Kropina metric, and other metrics related to the Zermelo problem with strong wind \cite{CJS14,JV13}.
\subsection{Conic pseudo-Finsler structures}
We recall that a \emph{pseudo-Finsler} structure (or a \emph{conic pseudo-Finsler} structure)
on a (connected) manifold $M$ consists of an open subset $\mathcal T\subset T_0M$,
where $T_0M$ denotes the tangent bundle with its zero section removed, and a smooth function $L:\mathcal T\to\mathds R$
satisfying the following properties:
\begin{enumerate}[(i)]
\item for all $p\in M$, the intersection $\mathcal T_p=\mathcal T\cap T_pM$ is a non-empty open cone of the tangent space $T_pM$;\smallskip

\item $L(tv)=t^2L(v)$ for all $v\in\mathcal T$ and all $t>0$;
\item for all $v\in\mathcal T$, the Hessian $g_v(u,w)=\frac{1}{2}\frac{\partial^2}{\partial t\partial s}L(v+tu+sw)|_{t=s=0}$ is nondegenerate, where $u,w\in T_{\pi(v)}M$.
\end{enumerate}
By continuity, the \emph{fundamental tensor} $g_v$ has constant index, which is called the index of the pseudo-Finsler structure.
The case when $\mathcal T=T_0M$ and the index of $g_v$ is zero, i.e., $g_v$ is positive definite for all $v$, is the standard Finsler structure.
When $g_v$ does not depend on $v$, then we have a standard pseudo-Riemannian manifold. As in the classical Finsler metrics, we can define the associated Chern connection as a family of affine connection (see for example \cite{Jav13}).
\subsection{Pseudo-Finsler isometries}
An \emph{isometry
of the pseudo-Finsler structure $(M,\mathcal T,L)$} is a diffeomorphism $f$ of $M$, with:
\begin{equation}\label{eq:defpsFinslisometry}
\mathrm df(\mathcal T)=\mathcal T,\quad\text{and}\quad
L\circ\mathrm df=L.
\end{equation}
The notion of local isometry is defined similarly.
Clearly, the set $\mathrm{Iso}(M,\mathcal T,L)$ of such pseudo-Finsler isometries is a group with respect to composition, and one has a natural action of $\mathrm{Iso}(M,\mathcal T,L)$ on $M$.
It is proved in \cite{GalPic14} that this action makes $\mathrm{Iso}(M,\mathcal T,L)$ into a Lie transformation group of $M$.
In Appendix~\ref{sec:dimisopseudoFinsler} we will show that $\mathrm{dim}\left(\mathrm{Iso}(M,\mathcal T,L)\right)\le \frac12n(n+1)$, where $n=\mathrm{dim}(M)$. By a similar argument, we will show here that the same inequality holds for the dimension of the space of germs of infinitesimal symmetries of a pseudo-Finsler manifold, see Proposition~\ref{thm:FpFadmissible}.
\smallskip

Given a pseudo-Finsler structure $(M,\mathcal T,L)$ of index $k$, there is an associated pseudo-Riemannian
metric of index $2k$ on the manifold $\mathcal T$, called the \emph{Sasaki metric} of $(M,\mathcal T,L)$,
that will be denoted by $g^L$, and
which is defined as follows. Let $\pi:TM\to M$ be the canonical projection.
The geodesic spray of $L$ defines a \emph{horizontal distribution} on
$\mathcal T$, i.e., a rank $n$ distribution on $\mathcal T$ which is everywhere transversal to the canonical
vertical distribution.
Equivalently, the horizontal distribution associated with $L$ can be defined using the Chern connection of
the pseudo-Finsler structure. In particular, a vector $X\in T \mathcal T$ is horizontal iff $D_{\pi(\alpha)}^\alpha \alpha(0)=0$, where $D$ is the covariant derivative induced by the Chern connection, and  $\alpha:(-\varepsilon,\varepsilon)\rightarrow \mathcal T$ is a curve (transversal to the fibers) such that $\dot\alpha(0)=X$.
For $p\in M$ and $v\in\mathcal T_p$, let us denote by $\mathrm{Ver}_v=\mathrm{Ker}(\mathrm d\pi_v)$ and $\mathrm{Hor}^F_v$ the corresponding subspaces of $T_v\mathcal T$. One has a canonical identification $i_v:T_pM\to \mathrm{Ver}_v$ (given by the differential at $v$ of the inclusion
$T_pM\hookrightarrow TM$); moreover, the restriction of $\mathrm d\pi_v:\mathrm{Hor}_v^L\to T_pM$ is an isomorphism. The Sasaki metric $g^L$ is defined by the following properties:
\begin{itemize}
\item on $\mathrm{Ver}_v$, $g^L$ is the push-forward of the fundamental tensor $g_v$ by the isomorphism
$i_v:T_pM\to\mathrm{Ver}_v$;
\item on $\mathrm{Hor}^L_v$, $g^L$ is the pull-back of the fundamental tensor $g_v$ by the isomorphism
$\mathrm d\pi_v:\mathrm{Hor}^L_v\to T_pM$;
\item $\mathrm{Ver}_v$ and $\mathrm{Hor}^L_v$ are $g^L$-orthogonal.
\end{itemize}
A (local) vector field $X$ on $M$ whose flow consists of (local) isometries for the pseudo-Finsler structure $(M,\mathcal T,L)$ will be called a (local) \emph{Killing field of $(M,\mathcal T,L)$}.
\subsection{Lifting of isometries and Killing fields}
\label{sec:liftingpseudoFinsler}
If $f$ is a (local) isometry of $(M,\mathcal T,L)$, then $\mathrm df$ is a local isometry of
the pseudo-Riemannian manifold $\mathcal T$ endowed with the Sasaki metric $g^L$. This is proved in
\cite{GalPic14}, see Appendix~\ref{sec:dimisopseudoFinsler}.

The lifting $f\mapsto \mathrm df$ of pseudo-Finsler isometries preserves the composition of diffeomorphisms. Thus, given a (local) vector field $X$ in $M$, the flow of $X$ can be lifted to a flow in $TM$; this is the flow of a (local) vector field $\widetilde X$ in $TM$. If $X$ is a Killing field of $(M,\mathcal T,L)$, then the flow of $\widetilde X$
preserves\footnote{This follows from the first property of pseudo-Finsler isometries in \eqref{eq:defpsFinslisometry}.}
$\mathcal T$, and we obtain (by restriction) a vector field on $\mathcal T$.
Clearly, the flow of $\widetilde X$ consists of $g^L$-isometries, so that $\widetilde X$ is a (local) Killing vector field of the pseudo-Riemannian manifold $(\mathcal T,g^L)$.
An explicit formula for $\widetilde X$ can be written using local coordinates and an auxiliary symmetric
connection on $M$.
If $F_t$ denotes the flow of $X$ and $\widetilde F_t=\mathrm dF_t$ its lifting, we have:
\[\widetilde X(v)=\tfrac{\mathrm d}{\mathrm dt}\big\vert_{t=0}\widetilde F_t(v)=\tfrac{\mathrm d}{\mathrm dt}\big\vert_{t=0}\mathrm dF_t(v)\]
for all $v\in\mathcal T$ and all $x\in M$. The curve $t\mapsto\mathrm dF_t(v)$ projects onto the curve $t\mapsto F_t(x)$ in $M$, so that the horizontal component of $\widetilde X(v)$ is:
\begin{equation}\label{eq:horcomptildeX}
\big[\widetilde X(v)\big]_{\mathrm{hor}}=X(x).
\end{equation}
Using the auxiliary connection,
let us compute the vertical component of $\widetilde X(v)$ by:
\begin{equation}\label{eq:vertcomponentfield}
\big[\widetilde X(v)\big]_{\mathrm{ver}}=\tfrac{\mathrm D}{\mathrm dt}\big\vert_{t=0}\big[\mathrm dF_t(v)\big]=\nabla_vX.
\end{equation}

\subsection{The sheaf of local pseudo-Finsler Killing fields}
Let $\mathcal F^\text{pF}$ denote the sheaf of local Killing fields of $(M,\mathcal T,L)$.
Using the relations between the pseudo-Finsler structure and the associated Sasaki metric, we can prove the following:
\begin{prop}\label{thm:FpFadmissible} Let $(M,\mathcal T,L)$ be a pseudo-Finsler manifold, with $n=\mathrm{dim}(M)$. Then:
\begin{itemize}
\item[(a)] $\mathcal F^\text{pF}$ has bounded rank: for all open connected subset $U\subset M$
\[\mathrm{dim}\big(\mathcal F^\text{pF}(U)\big)\le\tfrac12n(n+1);\]
\item[(b)] $\mathcal F^\text{pF}$ has the unique continuation property.
\end{itemize}
\end{prop}
\begin{proof}
Denote by $\mathcal F^{\mathrm{pR}(g^L)}$ the sheaf of local Killing fields of the pseudo-Riemannian
manifold $(\mathcal T,g^L)$, where $g^L$ is the Sasaki metric of $(M,\mathcal T,L)$.
Given a connected open subset $U\subset M$, denote by $\widetilde U=\pi^{-1}(U)\cap\mathcal T$, where
$\pi:TM\to M$ is the canonical projection. Then, $\widetilde U$ is a connected open subset of $\mathcal T$.
Given $X\in\mathcal F^\text{pF}(U)$, let $\widetilde X\in\mathfrak X(\widetilde U)$ be the lifting of $X$.
As we have seen in Subsection~\ref{sec:liftingpseudoFinsler}, $\widetilde X\in\mathcal F^{\mathrm{pR}(g^L)}(\widetilde U)$, and the maps $X\mapsto\widetilde X$ gives an injective linear map
from $\mathcal F^\text{pF}(U)$ to $\widetilde X\in\mathcal F^{\mathrm{pR}(g^F)}(\widetilde U)$.
Injectivity follows easily from the fact that, for $v\in\widetilde U$ and $x=\pi(v)$,
then the horizontal component of $\widetilde X(v)$ is $X(x)$. Then, rank boundedness and unique continuation property for the sheaf $\mathcal F^\text{pF}$ follow immediately from the corresponding properties of $\mathcal F^{\mathrm{pR}(g^L)}$.
\end{proof}
Observe that the Killing fields of pseudo-Finsler metrics are also characterized by the equation in part $(i)$ of Proposition \ref{thm:diffeqnFinslerKilling}. Then it can be shown that they are analytically determined as in Example \ref{exa:Killinganaldet}.
\begin{cor}\label{thm:extpseudoFinslerKillfields}
Let $(M,\mathcal T,L)$ be a simply connected pseudo-Finsler manifold. If $(M,\mathcal T,L)$ is either locally homogeneous, or real-analytic, then every local Killing field for $(M,\mathcal T,L)$ admits a global extension.
\end{cor}
\begin{proof}
By Proposition \ref{thm:FpFadmissible}, the sheaf of local Killing fields of $(M,\mathcal T,L)$ is admissible. The conclusion will follow from Theorem \ref{thm:globalregular}, once we show that the sheaf is regular. If $(M,\mathcal T,L)$ is locally homogeneous, then the regularity follows from the fact that the sheaf of Killing vectors is a full sheaf and Proposition \ref{thm:transitiveimpliesregular}. For the real-analytic case, we will apply Theorem \ref{thm:ADimpliesreg}. Consider any point $p\in M$ and an open (simply-connected) subset $U$ which contains $p$ and it admits a coordinate system. Now let $q\in U$ and $K$ a local Killing field of $(M,\mathcal T,L)$ defined in $U'\subset U$, with $q\in U'$ and $\tilde{K}$, the  lifting to $\mathcal T$ which is a local Killing field of $(\mathcal T,g^L)$. Observe that as the sheaf of local Killing fields of a pseudo-Riemannian manifold is admissible and regular (when all the data is analytic), then $\tilde{K}$ can be extended to a Killing field  $\hat{K}$ in $\tilde{U}=\pi^{-1}(U)$. Moreover, as such extension is projectable on an open subset of $\tilde{U}$, Lemma \ref{lemmaprojectable} ensures that $\tilde{K}$ is projectable on the entire $\tilde{U}$.
The projection is an analytic extension of $K$ to $U$, as required to apply Theorem \ref{thm:ADimpliesreg}.
\end{proof}

\section{Affine fields of sprays}
\label{sec:sprays}
Let us consider a manifold $M$ and a spray $S$ in a conic subset $A$ of the tangent bundle $TM$, which by definition  is a vector field in $A$ which,  in the natural coordinates  $(x^1,\ldots,x^n,y^1,\ldots,y^n, \Omega\times \R^n)$ for $TM$ associated with a coordinate system $(\varphi =(x^1,\ldots,x^n),\Omega)$ of $M$, is expressed as
\[
S=\sum_{i=1}^n(y^i \frac{\partial}{\partial x^i}-2 G^i\frac{\partial}{\partial y^i})
\]
where $G^i$ are positive homogeneous functions of degree two on $T\Omega\cap A$ (see \cite[Chapter 4]{Sh01}). The {\it geodesics of the spray} are the projections in $M$ of the integral curves of $S$ in $A$ and an {\it affine vector field of the spray} is a vector field $K\in {\mathfrak X}(M)$ whose flow preserves the geodesics of the spray. Associated with any spray, we can define the {\it Berwald connection,} which is determined by the {\it Christoffel symbols}
\[\Gamma_{\,\,jk}^i(x,y):=\frac{\partial^2 G^i}{\partial y^j\partial y^k}(x,y)\]
moreover, the {\it non-linear connection} is determined by
\[N_{\,j}^i(x,y):=\frac{\partial G^i}{\partial y^j}(x,y).\]
With the help of the Christoffel symbols, which are positive homogeneous of degree zero, for every {\it $A$-admissible vector field } $V$ (here $A$-admissible  means that at every point $p\in M$ takes values in $A_p=A\cap T_pM$), we can define a linear connection $\nabla^V$ which is torsion-free, and for every curve $\alpha:[a,b]\rightarrow M$, and every admissible vector field $W$ along $\alpha$, namely, with $W(t)\in A$ for every $t\in [a,b]$, a covariant derivative $D^W_\alpha$. In fact, the geodesics of the spray are characterized by the equation $D^{\dot\alpha}_\alpha \dot\alpha=0$. Then we can also define the {\it Berwald tensor} as
\[B_V(X,Y,Z)=\frac{\partial}{\partial t}\left(\nabla^{V+tZ}_XY\right)|_{t=0},\]
where  $X,Y,Z$ are arbitrary smooth vector fields on $\Omega$ (see \cite[Ch. 6]{Sh01}). Let us observe that the tensor $B_V$ is symmetric in its first two arguments because  $\nabla^V$ is torsion-free. Moreover,  in coordinates $B_V$ is given by
\[B_V(X,Y,Z)=X^j Y^k Z^l B^i_{\,\,jkl}(V) \frac{\partial}{\partial x^i},\]
where $B^i_{\,\,jkl}(V)=\frac{\partial \Gamma^i_{\,\, jk}}{\partial y^l}(V)=\frac{\partial^3 G^i}{\partial y^j\partial y^k\partial y^l}$. In particular, it is clear  that the value of $B_V(X,Y,Z)$ at a point $p\in \Omega$ depends only on $V(p)$ and not on the extension $V$ used to compute it and $B_V$ is symmetric. From the homogeneity of $\nabla^V$, namely, from the property $\nabla^{\lambda V}=\nabla^V$ for every $\lambda>0$, it follows easily that
$B_V(X,Y,V)=0$ and by the symmetry:
\begin{equation}\label{Pprop}
B_V(V,X,Y)=B_V(X,V,Y)=B_V(X,Y,V)=0.
\end{equation}
Now observe that proceeding as in \cite{J14}, we can relate the curvature tensor of the affine connection $\nabla^V$ with the  Berwald curvature tensor. More precisely, proceeding  in an analogous way to \cite[Theorem 2.1]{J14}, we get
\begin{equation}
R^V(X,Y)Z=R_V(X,Y)Z+B_V(Y,Z,\nabla^V_XV)-B_V(X,Z,\nabla^V_YV),
\end{equation}
where
\begin{multline*}
R_V(X,Y)Z=\left[Z^iY^j(X^p \frac{\partial\Gamma_{\,\,ij}^k}{\partial x^p}(V)-X^m N_{\, m}^l(V) \frac{\partial\Gamma_{\,\,ij}^k}{\partial y^p}(V)) \right.\\
-Z^iX^j (Y^p\frac{\partial \Gamma_{\,\,ij}^k}{\partial x^p}(V)-Y^m N_{\, m}^l(V) \frac{\partial\Gamma_{\,\,ij}^k}{\partial y^p}(V))\\\left.
 +Z^iY^j X^m \left(\Gamma_{\,\,ij}^l(V) \Gamma_{\,\,lm}^k(V)-\Gamma_{\,\,im}^l (V)\Gamma_{\,\,lj}^k (V)\right)\right]\frac{\partial}{\partial x^k}.\label{Firstcur}
 \end{multline*}
Let us consider an  $A$-admissible two parameter map
\[
  \Lambda: [a,b]\times (-\varepsilon,\varepsilon)\rightarrow M,\quad   (t,s)\rightarrow \Lambda(t,s).
\]
(Here `$A$-admissible' means that $\partial_t \Lambda(t,s)=\frac{\partial\Lambda}{\partial t}(t,s)\in A$ for every $(t,s)\in [a,b]\times (-\varepsilon,\varepsilon)$). Moreover, denote $\beta_t:(-\varepsilon,\varepsilon)\rightarrow M $ defined as $\beta_t(s)=\Lambda(t,s)$ for every $t\in [a,b]$ and $\gamma_s:[a,b]\rightarrow M$ defined as $\gamma_s(t)=\Lambda(t,s)$ for every $s\in (-\varepsilon,\varepsilon)$. Then for every smooth vector field $W$ along $\Lambda$ we define the curvature operator
\begin{equation}\label{def:curtwopar}
R^\Lambda(W):=D_{\gamma_{s}}^{\partial_t\Lambda}D_{\beta_{t}}^{\partial_t\Lambda}W-D_{\beta_{t}}^{\partial_t\Lambda}D_{\gamma_{s}}^{\partial_t\Lambda}W,
\end{equation}
and proceeding as in \cite[Theorem 1.1]{J14} and taking into account \eqref{Pprop}, we get
\begin{equation}\label{curformula}
R^\Lambda(W)=R_{\dot\gamma}(\dot\gamma,U)W+B_{\dot\gamma}(U,W,D_\gamma^{\dot\gamma}\dot\gamma),
\end{equation}
where $U$ is the variational vector field of $\Lambda$ along $\gamma$, namely, $U(t)=\partial_s\Lambda(t,0)$. We can now characterize affine fields of sprays.
\begin{prop}\label{thm:charaffinefields}
A vector field $K\in {\mathfrak X}(M)$ is an affine field of a spray $(A,S)$ if and only if for any geodesic $\gamma:[a,b]\rightarrow M$ of the spray, it holds
\begin{equation}\label{Jacobi}
 D^{\dot\gamma}_\gamma D^{\dot\gamma}_\gamma K+R_{\dot\gamma}(\dot\gamma,K)\dot\gamma=0.
 \end{equation}
\end{prop}
\begin{proof}
If $K$ is an affine field and $\varphi_s$ is its flow, consider the two-parameter map $\Lambda(t,s)=\varphi_s(\gamma(t))$, which is defined in $[a,b]\times (-\varepsilon,\varepsilon)$ for $\varepsilon>0$ small enough. Observe that for every $s$, the curve $\gamma_s$ is a geodesic and then $D_{\gamma_s}^{\dot\gamma_s}\dot\gamma_s=0$. Therefore, $D_{\beta_t}^{\dot\gamma_s} D_{\gamma_s}^{\dot\gamma_s}\dot\gamma_s=0$ and using \eqref{def:curtwopar} and \eqref{curformula}, we get
\[R_{\dot\gamma_s}(\dot\gamma_s,K)\dot\gamma_s=-D_{\gamma_{s}}^{\dot\gamma_s}D_{\beta_{t}}^{\dot\gamma_s}\dot\gamma_s.\]
As the Berwald connection is torsion-free, $D_{\beta_{t}}^{\dot\gamma_s}\dot\gamma_s=D_{\gamma_{s}}^{\dot\gamma_s}\dot\beta_t$ and \eqref{Jacobi} follows.

For the converse, assume that $K$ is a vector field that satisfies \eqref{Jacobi} for any geodesic $\gamma:[a,b]\rightarrow M$. Let $t_0\in [a,b]$ and $\alpha:(-\varepsilon,\varepsilon)\rightarrow M$ an integral  curve of $K$ with $\alpha(0)=\gamma(t_0)$. Now obtain a vector field $W$ along $\alpha$ by translating with the flow of $K$ the vector $\dot\gamma(t_0)$ and consider the geodesic $\gamma_s$ starting at $\alpha(s)$ with velocity $W(s)$ for every $s\in (-\varepsilon,\varepsilon)$. We obtain in this way a two-parameter map
$\Lambda:[a,b]\times   (-\varepsilon,\varepsilon)\rightarrow M$ given by $\Lambda(t,s)=\gamma_s(t)$ for every $(t,s)\in [a,b]\times   (-\varepsilon,\varepsilon)$. As the longitudinal curves of the variation are geodesics, then the variation vector field along any $\gamma_s$ is a Jacobi field that satisfies \eqref{Jacobi}. Moreover, the vector $K$ induced a vector field along  $\Lambda$, which has to coincide with the variational vector field of $\Lambda$, because it is also a Jacobi field along every $\gamma_s$ with the same initial conditions in $t_0$ (remember that the Berwald connection is torsion-free and then $D_{\gamma_s}^{\dot\gamma_s}K=D_\alpha^{\dot\gamma_s}\dot\gamma_s$). This concludes that the flow of $K$ maps $\gamma$ into the $\gamma_s$ and then it preserves geodesics, as desired.
\end{proof}
\begin{prop}\label{thm:affineadmissible}
The sheaf of local affine vector fields of a spray is admissible.
\end{prop}
\begin{proof}
Observe that given a point $q\in M$, it is always possible to find a geodesic $\gamma:[0,1]\rightarrow M$ for the spray $S$ such that $\gamma(1)=q$ and  $p=\gamma(0)$ and $q$ are non-conjugate along $\gamma$ (see \cite[Remark 3.2]{JaPi06a}). Let us see first that the sheaf of local affine vector fields satisfies the unique continuation property. Assume that $U$ is a connected open subset of $M$ that admits an affine field $K$ which is zero in an open subset $V\subset U$. If  $K$ is not zero in all $U$, then as $U$ is connected there exists a point $q\in U$ in the boundary of the zeroes of $K$. Choosing a geodesic $\gamma$ as above with image in $U$, as $p$ and $q$ are non-conjugate, we can find open subsets $\tilde{V}\subset U$ of $q$ and $\hat{V}\subset T_pM$ such that  $\exp_p:\hat{V}\rightarrow \tilde{V}$ is a diffeomorphism. Observe that $K$ restricted to geodesics is a Jacobi field, and recall that a Jacobi field of  a geodesic $\sigma:[a,b]\rightarrow M$  is determined by $K(\dot\sigma(t_0))$ and $\nabla_{\dot\sigma}^{\dot\sigma}K(t_0)$ with $t_0\in [a,b]$. Moreover, the operator $A_p\ni v\rightarrow \nabla_v^vK$ can be expressed in terms of a frame $E_1,E_2,\ldots,E_n$ in an open subset of $p$ as
\begin{equation}\label{nablaK}
\nabla_v^v K=\sum_{i,j=1}^n(v^i E_i(K^j) E_j+v^i K^j \nabla^v_{E_i}E_j),
\end{equation}
with $v=\sum_{i=1}^n v^i E_i$ and $K=\sum_{i=1}^n K^i E_i$. As $K$ is zero in an open subset of $\tilde{V}$ (this is because $q$ is in the boundary of the set of zeroes of $K$), it follows that $K(p)=0$ and $\nabla_{v_i}^{v_i}K (p)=0$ for $v_1,v_2,\ldots,v_n$ linearly independent vectors in $\tilde{V}$ (which is always possible because $\tilde{V}$ is open). We can then construct a linear system of equations using \eqref{nablaK} for $v_i$, $i=1,\ldots,n$:
\[ \sum_{i=1}^n v_l^i E_i(K^j)=0\]
where $v_l=\sum_{i=1}^nv_l^i E_i$ and $l,j=1,\ldots,n$. As $v_1,\ldots,v_n$ are linearly independent, the square matrix $\{v_l^i\}$ has rank equal to $n$ and then, we conclude that $E_i(K^j)=0$ for $i,j=1,\ldots,n$ in the point $p$. This implies
using \eqref{nablaK} that $A_p=0$ for every admissible $v$ and then that $K$ is zero in all $\tilde{V}$ contradicting that $q$ belongs to the boundary of the zeroes of $K$.

 Let us see now that it is bounded rank. Given an open subset $U$ and an affine field $K$ on $U$, choose as above a geodesic $\gamma:[0,1]\rightarrow U$ with $p=\gamma(0)$ and $q=\gamma(1)$ non-conjugate along $\gamma$, $V\subset U$ a neighborhood of $q$ and $\hat{V}\subset T_pM$ such that $\exp_p:\hat{V}\rightarrow V$ is a diffeomorphism.  Now observe that $K(p)$ and the operator $A_p$ (defined above) determine the value of $K$ in $V$, and by the unique continuation property in all $U$. The identity \eqref{nablaK} implies that $K$ and $A_p$ are determined by $K^i$, $E_i(K^j)$, $i,j=1,\ldots,n$ in any frame $E_1,E_1,\ldots,E_n$ and then the dimension of affine fields in $U$ is bounded by $n^2+n$.

\end{proof}
Given a vector field $X\in {\mathfrak X}(U)$ we will denote by $X^c$, its natural lift to $TU$. The following results is well-known (see \cite{SziToth11}).
\begin{prop}
A  vector field $X\in {\mathfrak X}(U)$ is affine for the spray manifold $(M,S)$ if and only if $[X^c,S]=0$.
\end{prop}
In particular, the last proposition implies that the sheaf of affine fields is analytically determined. Indeed, assume that $TU\subset TM$ admits an analytic Riemannian metric $g$ and an analytic orthogonal frame $E_1,E_2,\ldots, E_{2n}$. Then the collection ${\mathcal A}(U)$ of functions consists in one function  $F:\mathfrak X^\omega(U)\to C^\omega(U)$ defined as $F(X)=\sum_{i=1}^{2n} g([X^c,S],E_i)^2$.

\smallskip

 For our main result on this section, two additional ingredients are needed. On the one hand, the following lemma which proof is a straightforward consequence of the fact that the exponencial map is a local diffeomorphism in conjugate points.
%
\begin{lemma}\label{previouslemma}
Let $(M,S)$ be a spray and assume that $\gamma:[0,1]\rightarrow M$ is a geodesic of $(M,S)$ without self-intersections and such that $\gamma(0)$ does not have any conjugate point along $\gamma$. Then for every $t_0\in (0,1)$, there exists a convex open subset $\Omega\subset T_{\gamma(0)}M$ where the exponencial map $\exp_{\gamma(0)}$ is defined and it is a diffeomorphism onto its image, which contains the segment $\gamma|_{[t_0,1]}$.
\end{lemma}

 On the other hand, we also need the notion of reverse spray. Given $(M,S)$ a spray over $A$, we define the {\em reverse spray} $(M,\tilde{S})$ over $\tilde{A}=\{v\in TM: -v\in A\}$ as the spray whose coefficients are $\tilde{G}^i(x,y)=G^i(x,-y)$, being $G^i$ the spray coefficients of $S$ (in some natural coordinate system of $TM$). Geodesics of $(M,\tilde{S})$ are exactly the geodesics of $(M,S)$ with reverse parametrization.
%

\begin{prop}\label{regularaffine}
The sheaf of local affine vector fields of an analytic spray is regular.
\end{prop}
\begin{proof}
 We aim to apply Theorem \ref{thm:ADimpliesreg}. It is enough to show that for every point $q\in M$, there exists a neighborhood $U$ of $q$ such that every affine vector field $K$ defined in an open subset $V\subset U$ extends  analytically   to
a vector field $\tilde{K}$ in $U$. Given $q\in M$,  consider first a geodesic $\gamma:[0,1]\rightarrow M$ without self-intersections  such that $q=\gamma(1)$  and there is no conjugate point to $q$ along $\gamma$, denote $p=\gamma(0)$. Let $U$ be an open subset of $M$ such that there exists $\tilde{U}\subset T_pM$ with the restriction $\exp_p:\tilde{U}\rightarrow U$ being a diffeomorphism and such that every $u\in U$ does not have any conjugate point along the geodesic $\gamma_u:[0,1]\rightarrow M$ defined as $\gamma_u(t)=\exp_p(t \exp_p^{-1}(u))$  (observe that \cite[Remark 3.2]{JaPi06a} allows us to ensure the existence of $\gamma$ and $U$ by continuity of solutions of differential equations) and $\gamma_u$ does not have self-intersections. Now, if $K$ is an affine vector field defined in an open subset $V\subset U$, choose any $\tilde{p}\in V$,  and denote $\sigma=\gamma_{\tilde{p}}$. Observe that $\tilde{\sigma}(t)=\sigma(1-t)$ is a geodesic of the reverse spray $(M,\tilde{S})$ without conjugate points to $\tilde\sigma(0)=\tilde{p}$ and without self-intersections. Then we can apply Lemma \ref{previouslemma} for some $t_0$ such that $\tilde{\sigma}(t_0)\in V$, obtaining  open subsets $\Omega\subset T_{\tilde{p}}M$ and  $\tilde{\Omega}\in M$ such that  $\Omega$ is convex, $\tilde{\Omega}$ contains $\tilde\sigma(t_0)$ and $p$, and $\tilde{\exp}_{\tilde{p}}:\Omega\rightarrow \tilde{\Omega}$ is a diffeomorphism. We can assume that the image by $\tilde{\exp}_{\tilde{p}}$ of the segments $\{t u\in T_{\tilde{p}}M: t\in R\}\cap \Omega$ intersects $V$ by making $\Omega$ (and $\tilde{\Omega}$) smaller if necessary.   Recall that the restriction of $K$ to geodesics is a Jacobi field (Proposition \ref{thm:charaffinefields}). We can now extend $K$ to a vector field $\tilde{K}$ defined in $\tilde{\Omega}$ by using Jacobi fields along geodesics of $(M,\tilde{S})$ departing from $\tilde{p}$.  By Lemma \ref{thm:anextF}, $\tilde{K}$ is an affine field. In this way, we get a vector field in a neighborhood of $p$. Now we can define an analytic vector field on $U$ using the exponencial map $\exp_p$, Jacobi fields and the initial conditions provided by this extension, which coincides with the initial $K$ in a neighborhood of $\tilde{\sigma}(t_0)$ (just consider the geodesics from $p$ that remain in $\tilde{\Omega}$). This concludes the proof.

\end{proof}
\appendix

\section{Dimension of the group of pseudo-Finsler isometries}
\label{sec:dimisopseudoFinsler}
Let $(M,\mathcal T,L)$ be a pseudo-Finsler structure of index $k$, as defined in Section~\ref{sec:pseudoFinsler}, and let
$g^L$ be the associated Sasaki metric on $\mathcal T$.

It is proved in \cite[Lemma~1]{GalPic14} that, given any $C^2$-diffeomorphism $f:M\to M$, one
has the following commutative diagrams:
\begin{equation}\label{eq:2commdiagramas}
\xymatrix{T_v(TM)\ar[rr]^{\mathrm d^2f}\ar[d]_{\mathrm d\pi}&&T_{\mathrm df(v)}(TM)\ar[d]^{\mathrm d\pi}\cr T_pM\ar[rr]_{\mathrm df}&&T_{f(p)}M}
\qquad \xymatrix{T_pM\ar[rr]^{\mathrm df}\ar[d]_{\mathfrak i_v}&&T_{f(p)}M\ar[d]^{\mathfrak i_{\mathrm df(v)}}\cr \mathrm{Ver}_v\ar[rr]_{\mathrm d^2f}&&\mathrm{Ver}_{\mathrm d^2f(v)}}
\end{equation}
Using this diagram, it is proved in \cite{GalPic14} that if $f$ is an isometry of $(M,\mathcal T,L)$,
then $\mathrm df:\mathcal T\to\mathcal T$ is an isometry of $g^L$.
\smallskip

We have the canonical upper bound on the dimension of the isometry group of a pseudo-Finsler structure.

The map $f\mapsto\mathrm df$ gives a natural injection of $\mathrm{Iso}(M,\mathcal T,L)$ into the isometry group of the Sasaki pseudo-Riemannian metric $g^L$. Since $\mathrm{dim}(\mathcal T)=2n$, then we have
$\mathrm{dim}\left(\mathrm{Iso}(M,\mathcal T,L)\right)\le n(2n+1)$, where $n=\mathrm{dim}(M)$.
However, the following sharper estimate holds:
\begin{prop}
Let $(M^n,\mathcal T,L)$ be a  pseudo-Finsler manifold. Then, the Lie group $\mathrm{Iso}(M,\mathcal T,L)$ has dimension less than or equal to $\frac12n(n+1)$.
\end{prop}
\begin{proof}
Using \eqref{eq:2commdiagramas}, one sees that
using the identifications $i_v:T_pM\stackrel\cong\longrightarrow\mathrm{Ver}_v$ and $\mathrm d\pi_v:\mathrm{Hor}_v^F\stackrel\cong\longrightarrow T_pM$,
for all $f\in\mathrm{Iso}(M,\mathcal T,L)$, both restrictions of $\mathrm d(\mathrm df)$ to $\mathrm{Ver}_v$ and to $\mathrm{Hor}_v^F$ of $\mathrm d(\mathrm df)_v$ coincide with $\mathrm df_v:T_pM\to T_{f(p)}M$.
This gives  $\mathrm{dim}\big(\mathrm{Iso}(M,\mathcal T,L)\big)\le\frac12n(n+1)$.
\end{proof}

\section{Real-analytic compatible connections}\label{app:RAcompconnections}
In this appendix we give some detail on the proof that real-analytic $G$-structures on a real-analytic differential manifold admit a real-analytic compatible connection.
\begin{lemma}\label{thm:RAsections}
Let $\pi:E\to M$ be a real-analytic fiber bundle on a real-analytic differentiable manifold. If $\pi$ admits a smooth section, then it admits a real-analytic section.
\end{lemma}
\begin{proof}
Let $s$ be a smooth section of $\pi$. By \cite[Theorem 5.1, p.\ 65]{hirsch76}, $C^\omega(M,E)$ is dense in $C^\infty(M,E)$ in the strong Whitney topology. Thus, $s$ admits real-analytic approximations in $C^\infty$. Clearly, if $s'$ is a real-analytic approximation of $s$, $s'$ need not be a section, i.e.,
$\phi=\pi\circ s'$ need not be the identity. However, such a real-analytic approximation $s'$ can be found such that $\pi\circ s'$ is a (real-analytic)
diffeomorphism. This follows from \cite[Theorem 1.7, p.\ 38]{hirsch76}, which says that the set of $C^\infty$-diffeomorphisms is open in the set
$C^\infty$-endomorphism of $M$. The composition $s'\circ\phi^{-1}$ is the desired real-analytic section of $\pi$.
\end{proof}

\begin{cor}\label{thn:RAcompconnections}
Let $M^n$ be a real-analytic manifold, $G\subset\mathrm{GL}(n)$ a Lie subgroup, and let $\mathcal P\subset\mathrm{FR}(TM)$ be a real-analytic $G$-structure on $M$. Then, there exists a real-analytic connection $\nabla$ on $M$ which is compatible with $\mathcal P$.
\end{cor}
\begin{proof}
This is a direct application of Lemma~\ref{thm:RAsections}, observing that connections in the $G$-principal bundle $\mathcal P\to M$
correspond to sections of a certain fiber bundle $\pi:E\to M$ that has the same regularity as $\mathcal P$. The definition of this fiber bundle is given by the construction below. First, one consider the affine bundle
$\mathcal E\to\mathcal P$ over $\mathcal P$, whose fiber $\mathcal E_p$ at the point $p\in\mathcal P$ consists of all linear maps
$L:T_p\mathcal P\to\mathfrak g$ whose restriction to the vertical space coincides with the canonical isomorphism (here $\mathfrak g$ is the Lie algebra of $G$).
There is a left-action of $G$ on the fibers of $\mathcal E$, defined as follows:
the action of an element $g\in G$ on a linear transformation $L:T_p\mathcal P\to\mathfrak g$ is given by the composition of $L$:
\begin{itemize}
\item with the differential of the action of $g$ in $\mathcal P$, on the right;
\item with the adjoint $\mathrm{Ad}_g$, on the left.
\end{itemize}
Connection forms on $\mathcal P$ are, by definition, equivariant sections of the affine bundle $\mathcal E\to\mathcal P$, see for instance 
\cite[\S~2.2]{PicTau}. In turn, connection forms on $\mathcal P$ correspond to principal connections in the bundle $\mathcal P\to M$
(see \cite[Theorem 2.2.6]{PicTau}~), where the horizontal distribution of a principal connection is given by the kernel of a connection form.
One has a commutative diagram:
\[\xymatrix{\mathcal E\ar[rr]^{g}\ar[d]_{\pi}&&\mathcal E\ar[d]^{\pi}\cr \mathcal P\ar[rr]_{g^{-1}}&&\mathcal P;}\]
in this situation, the quotient $E=\mathcal E/G$ is a fiber bundle over $M$, whose sections correspond to equivariant
sections of $\mathcal E\to\mathcal P$. which give $G$-principal connections in $\mathcal P$, correspond to
sections of $E\to M$.
A real-analytic section of $E$ is a real-analytic connection compatible with $\mathcal P$.
\end{proof}

\end{document}